\documentclass[a4paper,10pt, oneside]{article}
\usepackage{amsmath,amssymb,amsfonts,amsthm,mathrsfs,graphicx}
\usepackage[affil-it]{authblk}
\usepackage{authblk}
\usepackage[inline]{enumitem}
\usepackage[font=small,labelfont=md,textfont=it]{caption}
\usepackage{floatrow,float}
\usepackage[titletoc, title]{appendix}
\usepackage[colorlinks,linkcolor=blue,citecolor=blue]{hyperref}
\usepackage{etoolbox}
\usepackage{longtable}
\usepackage{diagbox}
\usepackage{pifont}
\usepackage{booktabs,makecell,multirow}
\usepackage[capitalise,nosort]{cleveref}
\usepackage{cases,color}
\usepackage{ulem}
\usepackage{pifont}
\usepackage{siunitx}
\usepackage{algorithm,algorithmicx,algpseudocode}
\usepackage{subfigure}
\crefname{equation}{}{}
\crefname{lemma}{Lemma}{Lemmas}
\crefname{theorem}{Theorem}{Theorems}
\crefname{discr}{Discretization}{Discretizations}
\crefname{assumption}{Assumption}{Assumptions}
\numberwithin{equation}{section}
\apptocmd{\sloppy}{\hbadness 10000\relax}{}{}

\newcommand{\ssnm}[1]
{
  \left\vert\kern-0.25ex
  \left\vert\kern-0.25ex
  \left\vert
  {#1}
  \right\vert\kern-0.25ex
  \right\vert\kern-0.25ex
  \right\vert
}

\makeatletter
\def\spher@harm#1{%
  \vbox{\hbox{%
    \offinterlineskip
    \valign{&\hb@xt@2\p@{\hss$##$\hss}\vskip.2ex\cr#1\crcr}%
  }\vskip-.36ex}%
}
\def\gshone{\spher@harm{.}}
\def\gshtwo{\spher@harm{.&.}}
\def\gshthree{\spher@harm{.&.&.}}
\let\gsh\spher@harm
\makeatother

\newtheorem{lemma}{Lemma}[section]
\newtheorem{remark}{Remark}[section]
\newtheorem{theorem}{Theorem}[section]
\newtheorem{example}{Example}[section]

\makeatletter\def\@captype{table}\makeatother

\begin{document}
	
\title{
  \Large\bf A fast fully discrete mixed finite element  scheme for fractional viscoelastic models of wave propagation\thanks
  {
    This work was supported in part by National Natural Science Foundation
    of China (12171340).
  }
}
\author{
  Hao Yuan \thanks{Email:  787023127@qq.com},
  Xiaoping Xie \thanks{Corresponding author. Email: xpxie@scu.edu.cn} \\
  {School of Mathematics, Sichuan University, Chengdu 610064, China}
}

\date{} 
\maketitle

\begin{abstract}

Due to the nonlocal feature of fractional differential operators,  the numerical solution to  fractional partial differential equations usually requires  expensive  memory and computation    costs. 
This paper  develops  a fast scheme for   fractional viscoelastic models of wave propagation. 
 We first apply the Laplace transform to convert  the time-fractional constitutive equation into an integro-differential form that involves   the Mittag-Leffler function as a convolution kernel. Then we construct   an efficient sum-of-exponentials (SOE) approximation for the Mittag-Leffler function.  We use mixed finite elements for the spatial discretization and the Newmark scheme for the temporal discretization of the    second  time-derivative of the displacement variable in the kinematical equation and finally obtain the fast algorithm.      Compared with the traditional L1 scheme for time fractional derivative, our fast scheme reduces the memory complexity from  $\mathcal O(N_sN) $ to $\mathcal O(N_sN_{exp})$ and the computation complexity from $\mathcal O(N_sN^2)$ to $\mathcal O(N_sN_{exp}N)$, where $N$ denotes   the total number of temporal grid points, $N_{exp}$   the number of exponentials in SOE, and    $N_s$      the complexity of memory and computation related to the spatial discretization.   
Numerical experiments confirm the theoretical results.  

 \end{abstract}


\section{Introduction}
Assume that  $\Omega\subset \mathbb{R}^d$($d=2$ and $3$) is a bounded open domain with boundary $\partial\Omega$,     $T>0$ is the time length, and $\alpha\in (0,1)$ is a constant. 
Consider the following fractional viscoelastic model of wave propagation:
\begin{equation}\label{model}
	\left\{
	\begin{array}{ll}
		\rho \textbf{u}_{tt}-\mathrm{\textbf{div}}\sigma=\textbf{f}, & (x,t)\in\Omega\times[0,T], \\
		\sigma+\tau_{\sigma}^{\alpha}\frac{\partial^{\alpha}\sigma}{\partial t^{\alpha}}=\mathbb{D}(\varepsilon(\textbf{u})+\tau_{\varepsilon}^{\alpha}\frac{\partial^{\alpha}\varepsilon(\textbf{u})}{\partial t^{\alpha}}), &(x,t)\in\Omega\times[0,T],\\
		\textbf{u}=0, &  (x,t)\in\partial\Omega\times[0,T], \\
		\textbf{u}(x,0)=\textbf{u}_0,\textbf{u}_t(x,0)=\textbf{v}_0,\sigma(x,0)=\sigma_0,& x\in\Omega.
	\end{array}
	\right.
\end{equation}
Here  $\textbf{u}=(u_1,...,u_d)^T$ is the displacement field, $\sigma=(\sigma_{ij})_{d\times d}$ the symmetric stress tensor, $\textbf{div}\sigma=(\sum\limits_{i=1}^d \partial_i\sigma_{i1}, \cdots, \sum\limits_{i=1}^d \partial_i\sigma_{id} )^T$, $\varepsilon(\textbf{u})=(\bigtriangledown \textbf{u}+(\bigtriangledown \textbf{u})^{\mathrm{T}})/2$ the strain tensor,   $\tau_{\sigma}$ the relaxation time,   $\tau_{\varepsilon}$ the retardation time, $\rho(x)$   the mass density,  and  $\mathbb{D}$ the fourth order symmetric tensor. $\textbf{f}=(f_1,...,f_d)$ is the body force and  $\textbf{u}_0(x),\ \textbf{v}_0(x),\ \sigma_0(x)$ are initial data.  For any function $\textbf{v}(x,t)$,  denote $\textbf{v}_{t}:=\partial \textbf{v}/\partial t $ and  $\textbf{v}_{tt}:=\partial^2 \textbf{v}/\partial t^2 $, and for $0<\alpha<1$,  let $\frac{\partial^{\alpha} \textbf{v}}{\partial t^{\alpha}}$ be the $\alpha$-order  Caputo fractional derivative of $\textbf{v}(x,t)$    defined by
\begin{align}
	\frac{\partial^{\alpha} \textbf{v}}{\partial t^{\alpha}}(x,t)=\frac{1}{\Gamma(1-\alpha)}\int_{0}^{t}\frac{\textbf{v}_t(x,s)}{(t-s)^{\alpha}}ds.
\end{align}
We note that the following three   classical viscoelastic models correspond to different  choices of the  relaxation/retardation time in the  constitutive (second)  equation of  \eqref{model} with  $\alpha=1$: the Kelvin-Voigt model ($\tau_{\sigma}=0$, $\tau_{\varepsilon} >0$); the Maxwell model ($\tau_{\sigma}> 0$, $\tau_{\varepsilon}=0$) and the Zener model ($\tau_{\sigma}>0$, $\tau_{\varepsilon}>0$).

Many materials   display  elastic and viscous
kinematic   behaviours simultaneously. Such a feature, called viscoelasticity,  is commonly characterized by using springs, which obey the Hooke's law, and viscous dashpots, which obey the Newton's law.  Different combinations of the springs and dashpots lead to various viscoelastic models, e.g. the  Zener model, the Kelvin-Voigt model and the Maxwell model. We refer the reader to \cite{1960Bland,2007Dill,1998Drozdov,  Fung1966International,1988Boundary,1962Gurtin,2000Nonlinear,Marques2012Computational} for several monographs on the development  and application  of  the viscoelasticity theory. 

In recent decades,  fractional order differential operators, as   extension of integer order ones, have been widely used in many scientific and engineering fields such as physics, chemistry, materials science,  biology, finance and other sciences, due to their ability to accurately describe states or development processes with memory and hereditary   characteristics.  As far as  the viscoelastic materials with complex rheological properties are concerned,  more and more studies indicate that, comparing with the  integer order models,    time fractional viscoelastic models can  more precisely characterize   the creep and relaxation  dynamic behaviours and capture the effects of "fading" memory \cite{gemant1936method,gemant1950frictional,blair1944analytical,rabotnov1970creep,caputo1971linear,caputo1971new,bagley1981fractional,bagley1983theoretical,mainardi2022fractional}.

%
%

There are some works in the literature on the numerical analysis  of time fractional viscoelastic models. In \cite{enelund1997time} Enelund and Josefson  rewrote the constitutive equation of fractional Zener model (Riemann Liouville type) as an integro-differential equation with a weakly singular convolution kernel by Laplace transform and carried out finite element simulation. Based on the integro-differential form of constitutive equation from \cite{enelund1997time},  Adolfsson et al. \cite{adolfsson2004adaptive} proposed a piecewise constant discontinuous Galerkin method for a fractional order (Riemann Liouville type) viscoelastic differential equation. Subsequently, they applied  a discontinuous Galerkin method in time and a continuous Galerkin finite element method in space to discretize the quasi-static fractional viscoelastic model  \cite{adolfsson2008space}.  In \cite{yu2016fractional} Yu et al.  adopted finite element simulation for a fractional Zener model (Riemann Liouville type) with integro-differential form of constitutive equation in 3D cerebral arteries and aneurysms. Lam et al. \cite{lam2020exponential} presented a finite element scheme for 1D fractional Zener model (Caputo type) with integro-differential form of constitutive equation. In \cite{2024Liu-Xie} Liu and Xie proposed   a semi-discrete hybrid stress finite element method for a  time fractional viscoelastic model, where the corresponding integro-differential equation is of a Mittag-Leffler type convolution kernel, and derived error estimate for the semi-discrete scheme.


The nonlocal feature of fractional differential operators usually means  expensive  computational cost and memory cost in the  numerical simulation of fractional models. To tackle such difficulties, Lubich and Sch{\"a}dle \cite{lubich2002fast} proposed a new algorithm for the evaluation of convolution integral when  solving  wave propagation problems. The algorithm is based on local SOE approximation for the inverse Laplace transform of kernel function by applying trapezoidal rule to the contour integral. Li \cite{li2010fast} presented  a locally SOE approximation for the integral representation of the kernel function by using an efficient $Q$-point Gauss–Legendre quadrature.  Yu et al. \cite{yu2016fractional} considered an SOE approximation of Mittag-Leffler function by applying trapezoidal rule to the contour integral and applied it to the fractional Zener model. Jiang et al. \cite{jiang2017fast} and Yan et al. \cite{yan2017fast} split the convolution integral in the Caputo fractional derivative into a local part and  a history part, and presented   fast algorithms for time fractional diffusion equations  by adopting the SOE approximation (using Gauss-Jacobi quadrature and Gauss-Legendre quadrature) for the history part and L1 (L2-$1_\sigma$) formula for the local part. Baffet \cite{baffet2019gauss} divided the fractional integral of a function $f$ into a history term (convolution of the history of $f$ and a regular kernel) and a local term, and gave a method for fractional differential equations by using SOE approximation (by Gauss-Jacobi quadrature) for the history part and an implicit scheme for the local part. Zeng et al. \cite{zeng2018stable} developed a unified fast time-stepping method for both fractional integral and derivative operators by using truncated Laguerre-Gauss   quadrature for the kernel function in history part and a direct convolution method for local part.  In \cite{lam2020exponential} Lam et al. gave an SOE approximation (by Gauss-Legendre quadrature) for the integral representation of Mittag-Leffler function and applied it to a 1D fractional Zener model. 
We refer to \cite{bai2022numerical,baffet2017kernel,chen2019accurate,fang2020fast,2022Huang,jia2023fast,wang2018fast, yin2021class, zhang2022exponential} for some other fast algorithms for time fractional order PDEs. 

In this paper, we present an efficient numerical scheme for solving the fractional viscoelastic model \cref{model}. Our contribution lies in the following aspects.
\begin{itemize}
	\item The constitutive equation of model \cref{model} is converted to an integro-differential form with Mittag-Leffler function as the convolution kernel.
	\item An efficient SOE approximation (different from that of \cite{lam2020exponential}) is proposed for the Mittag-Leffler function and applied to accelerate the evaluation of the convolution. For a given tolerance error $\epsilon$ of the proposed SOE approximation, its  computation  complexity     is   $N_{exp}=\mathcal{O}(|\log\epsilon|^2)$. 
	\item An estimate of the truncation error of the SOE approximation is derived. We note that there is no truncation error estimation in \cite{lam2020exponential}.
	\item The proposed SOE approximation is applied to the fractional viscoelastic model to get a fast numerical scheme.
	\item The resulting fast scheme requires $\mathcal{O}(N_sN_{exp})$ memory complexity and $\mathcal{O}(N_s N_{exp}N)$ computation complexity,  in contrast to  $\mathcal{O}(N_sN)$  and  $\mathcal{O}(N_sN^2)$   for the traditional L1 scheme. Here $N$ denotes   the total number of temporal grid points and    $N_s$
   represents    the complexity of memory and computation related to the spatial discretization. In particular, if    the tolerance error of  the SOE approximation is taken as $\epsilon=\triangle t=T/N$,   we will have 
	$
		N_{exp}  =\mathcal{O}(\log^2N)
	$(cf. \cref{rmk2.4}).

\end{itemize}

The rest of this paper is arranged as follows. Section 2 introduces some preliminaries on the SOE approximation of Mittag-Leffler function. Section 3 gives two numerical schemes:  the   L1-Newmark scheme and the fast scheme with the SOE approximation. Finally, numerical examples are provided in Section 4 to verify the performance of the SOE approximation and the  fast scheme.

\section{Preliminary results}


\subsection{Alternative form of  the constitutive law and weak formulations}

We note that  the constitutive equation  in the model \cref{model}     is of the following   differential form:
\begin{equation}\label{constitutive}
	\sigma+\tau_{\sigma}^{\alpha}\frac{\partial^{\alpha}\sigma}{\partial t^{\alpha}}=\mathbb{D}(\varepsilon(\textbf{u})+\tau_{\varepsilon}^{\alpha}\frac{\partial^{\alpha}\varepsilon(\textbf{u})}{\partial t^{\alpha}}).
\end{equation}
  In this subsection we shall 
 convert it  to  an explicit expression of  $\sigma$ 
 when $\tau_{\sigma}\neq 0$. To this end, we first   introduce two basic tools: the Laplace transform and the Mittag-Leffler function.

Let $f(t)$ be a function defined in $\mathbb{R}^{+}$. The Laplace transform of $f(t)$ is defined by
\begin{equation}
	\hat{f}(s):={\mathcal L}\bigl( f(t) \,;\, s\bigr) =\int_{0}^{\infty}f(t)e^{-st}dt,
	\nonumber
\end{equation}
where $s\in\mathbb{C}$ and Re $s\geq0$. There holds the following property of the Laplace transform for the Caputo fractional derivative \cite{Podlubny1999,Diethelm2010}:
\begin{equation}\label{Laplace-property}
	\mathcal{L}\bigl( \frac{\partial^{\alpha} f}{\partial t^{\alpha}}(t) \,;\, s\bigr)=s^{\alpha}\hat{f}(s)-s^{\alpha-1}f(0), \qquad \alpha\in(0,1).
\end{equation}



 For $\alpha>0$, and $\beta\in\mathbb{R}$, the two-parameter Mittag-Leffler function is defined by 
\begin{equation*}
	E_{\alpha,\beta}(z):=\sum_{j=0}^{\infty}\frac{z^j}{\Gamma(j\alpha+\beta)},\quad z\in\mathbb{C}.
\end{equation*}
In particular, the one-parameter Mittag-Leffler function is given by
 $$E_{\alpha}(z):=E_{\alpha,1}(z)=\sum_{j=0}^{\infty}\frac{z^j}{\Gamma(j\alpha+1)}.$$
There hold the  following properties (cf. \cite{Diethelm2010,BangtiJin2021}):
\begin{lemma}
	(1) For $\alpha$, $\beta>0$ and $z\in\mathbb{C}$, there holds
	\begin{equation}\label{th2.2}
		E_{\alpha,\beta}(z)=zE_{\alpha,\alpha+\beta}(z)+\frac{1}{\Gamma(\beta)};
	\end{equation}
	(2) For $\alpha$, $\beta>0$, $\lambda\geq0$ and $t>0$, there holds
		\begin{equation}\label{LaplaceMLF}
			\begin{array}{ll}
				{\mathcal L}\bigl( t^{\beta-1}E_{\alpha,\beta}(-\lambda t^{\alpha}) \,;\, s\bigr)=\frac{s^{\alpha-\beta}}{s^\alpha+\lambda} \quad \bigl(Re \, s\geq0\bigr).
			\end{array}
		\end{equation}
\end{lemma}
It has been shown in \cite{2014Rogosin,meng2018green} that the following integral identity for the Mittag-Leffler function of $-t^{\alpha}$  holds   for $t>0$ and  $0<\alpha<1$: 
\begin{equation}\label{identityMLF}
		\begin{aligned}
			E_{\alpha}(-t^{\alpha})&=\frac{\sin(\alpha\pi)}{\pi}\int_{0}^{\infty}\frac{s^{\alpha-1}}{s^{2\alpha}+2s^\alpha\cos\alpha\pi+1}e^{-st}ds. 
		\end{aligned}
	\end{equation}

%
%

We are now at a position to derive the explicit expression of  $\sigma$ from the the constitutive equation \eqref{constitutive}. To begin with, we Laplace-transform   \eqref{constitutive} 
and apply \eqref{Laplace-property} to obtain
\begin{equation}
	\mathbb{D}^{-1}\left(\hat{\sigma}+\tau_{\sigma}^{\alpha}(s^{\alpha}\hat{\sigma}-s^{\alpha-1}\sigma_0)\right)=\varepsilon(\hat{\textbf{u}})+\tau_{\varepsilon}^{\alpha}(s^\alpha\varepsilon(\hat{\textbf{u}})-s^{\alpha-1}\varepsilon(\textbf{u}_0)),
\end{equation}
which yields
\begin{equation}\label{Laconstitutive}
	\begin{aligned}
		\mathbb{D}^{-1}\hat{\sigma}=&\frac{1+(\tau_{\varepsilon}s)^\alpha}{1+(\tau_{\sigma}s)^{\alpha}}\varepsilon(\hat{\textbf{u}})+\frac{s^{\alpha-1}}{1+(\tau_{\sigma}s)^{\alpha}}\left(\tau_{\sigma}^\alpha \mathbb{D}^{-1}\sigma_0+\tau_{\varepsilon}^\alpha\varepsilon(\textbf{u}_0)\right) \\
		=&(\frac{\tau_{\varepsilon}}{\tau_{\sigma}})^\alpha\cdot\frac{s^{\alpha-1}}{(\tau_{\sigma})^{-\alpha}+s^\alpha}\left(s\varepsilon(\hat{\textbf{u}})-\varepsilon(\textbf{u}_0)\right)+\frac{1}{\tau_{\sigma}^\alpha}\cdot\frac{s^{-1}}{(\tau_{\sigma})^{-\alpha}+s^\alpha}\left(s\varepsilon(\hat{\textbf{u}})-\varepsilon(\textbf{u}_0)\right)\\
		&+\frac{s^{\alpha-1}}{(\tau_{\sigma})^{-\alpha}+s^\alpha}\mathbb{D}^{-1}\sigma_0+\frac{1}{\tau_{\sigma}^\alpha}\cdot\frac{s^{-1}}{(\tau_{\sigma})^{-\alpha}+s^\alpha}\varepsilon(\textbf{u}_0).\\
	\end{aligned}
\end{equation}
Applying classical convolution theorem in Laplace transform \cite{Podlubny1999,Diethelm2010}, \cref{th2.2}, \cref{LaplaceMLF} and the inverse-Laplace-transform, we  finally  get the explicit expression
\begin{equation}\label{sigma}
	\begin{aligned}
		\mathbb{D}^{-1}\sigma=&\left((\frac{\tau_{\varepsilon}}{\tau_{\sigma}})^\alpha-1\right)\int_{0}^{t}E_{\alpha}(-(\frac{t-\tau}{\tau_{\sigma}})^\alpha)\varepsilon\bigl(\textbf{u}_t(\tau)\bigr)d\tau+\varepsilon(\textbf{u}_t)\\
		&+E_{\alpha}(-(\frac{t}{\tau_{\sigma}})^\alpha)(\mathbb{D}^{-1}\sigma_0-\varepsilon(\textbf{u}_0)).\\
	\end{aligned}
\end{equation}

In what follows we shall give a weak problem of  \cref{model} based on the alternative constitutive relation \cref{sigma}. 

Let ${L}^2(\Omega)$ be the space of  square integrable functions defined on $\Omega$, and let $\uline{L}^2(\Omega)$ and $\uuline{L}^2(\Omega)$ be its   vector and tensor analogues. We use $\langle \cdot \,,\, \cdot \rangle$ to denote the inner product on these three spaces. Define   
\begin{equation}
	\uuline{\mathrm{H}}(\mathrm{\textbf{div}},\Omega,S):=\{\chi=(\chi_{ij})_{d\times d}\in  \uuline{L}^2(\Omega)| \ \chi_{ij}=\chi_{ji}, \ \textbf{div}\chi \in \underline{L}^2(\Omega)\}.
	\nonumber
\end{equation}

In light of   \cref{sigma}, we have the following weak formulation for \cref{model}: Find $ \sigma\in \uuline{\mathrm{H}}(\mathrm{\textbf{div}},\Omega,S)$ and $\textbf{u}\in \uline{L}^2(\Omega)$ such that 
\begin{equation}\label{weak-form}
	\left\{
	\begin{array}{ll}
		\langle \rho\textbf{u}_{tt} \,,\, \textbf{v} \rangle-\langle \mathrm{\textbf{div}}\sigma \,,\, \textbf{v} \rangle=\langle \textbf{f} \,,\, \textbf{v} \rangle,\quad \forall\, \textbf{v}\in\uline{L}^2(\Omega),\\
		\langle \mathbb{D}^{-1}\sigma \,,\,\chi  \rangle+\big((\frac{\tau_{\varepsilon}}{\tau_{\sigma}})^\alpha-1\big)\int_{0}^{t}E_{\alpha}(-(\frac{t-\tau}{\tau_{\sigma}})^\alpha)\langle\mathrm{\textbf{div}}\chi \,,\,\textbf{u}_{t}(\tau)\rangle d\tau+\langle\mathrm{\textbf{div}}\chi \,,\,\textbf{u}_t\rangle \\
		=E_{\alpha}(-(\frac{t}{\tau_{\sigma}})^\alpha)\big(\langle\mathbb{D}^{-1}\sigma_0 \,,\,\chi\rangle+
		\langle\mathrm{\textbf{div}}\chi \,,\, \textbf{u}_0\rangle\big), \quad \forall\, \chi\in\uuline{\mathrm{H}}(\mathrm{\textbf{div}},\Omega,S),\\
		\textbf{u}(x,0)=\textbf{u}_0,\ \textbf{u}_t(x,0)=\textbf{v}_0.
	\end{array}
	\right.
\end{equation}

\begin{remark}
From the original 	model \cref{model}, we easily have the following weak formulation: Find $ \sigma\in \uuline{\mathrm{H}}(\mathrm{\bf{div}},\Omega,S)$ and ${\bf{u}}\in \uline{L}^2(\Omega)$ such that
{\rm\begin{equation}\label{weakformoftraditon}
	\left\{
	\begin{array}{ll}
		\langle\rho \textbf{u}_{tt} \,,\,\textbf{v}\rangle-\langle\mathrm{\bf{div}}\sigma \,,\,\textbf{v}\rangle=\langle\textbf{f} \,,\,\textbf{v}\rangle,\quad \forall\, {\textbf{v}}\in\uline{L}^2(\Omega),\\
		\langle\mathbb{D}^{-1}\sigma \,,\,\chi\rangle+\tau_{\sigma}^{\alpha}\langle\frac{\partial^{\alpha}\mathbb{D}^{-1}\sigma}{\partial t^{\alpha}} \,,\,\chi\rangle+\langle\mathrm{\bf{div}}\chi \,,\,\textbf{u}\rangle+\tau_{\varepsilon}^{\alpha}\langle\mathrm{\bf{div}}\chi \,,\,\frac{\partial^{\alpha}\textbf{u}}{\partial t^{\alpha}}\rangle=0, \  \forall\, \chi\in\uuline{\mathrm{H}}(\mathrm{\bf{div}},\Omega,S),\\
		\textbf{u} (x,0)=\textbf{u}_0,\ \textbf{u}_t(x,0)=\textbf{v}_0,\ \sigma(x,0)=\sigma_0.
	\end{array}
	\right.
\end{equation}}

	\end{remark}

%

\subsection{Efficient SOE approximation of Mittag-Leffler function}

Notice that  there is a   term  $E_{\alpha}(-(\frac{t-\tau}{\tau_{\sigma}})^\alpha)$  involved in in the weak formulation \eqref{weak-form}. As the Mittag-Leffler function is an infinite  series, how to compute such a term efficiently is crucial  to the design of fast algorithm for the fractional viscoelastic model \eqref{model}.
 
 In this   section we aim to construct an 
efficient sum-of-exponentials approximation of the Mittag-Leffler function $E_{\alpha}(-t^\alpha)$ based on the Gaussian quadrature   rule. .

\subsubsection{Gaussian quadrature approximation}\label{Gausssec}

	For a constant $l>1$, let $g(z)$ be  a function of one complex variable which is meromorphic in an open set containing the closure $\overline{B(l)}$ of the disc 
	\begin{equation}
		B(l)=\left\{z\in\mathbb{C}:|z|<l\right\} 
		\nonumber
	\end{equation} 
	and has only a finite number of simple poles $p_m$ in $B(l)$. 
	
 Consider the following Gaussian quadrature of $g(x)$ on interval $[-1,1]\subset(-l,l)$:
\begin{equation}\label{respole}
	\int_{-1}^{1}g(x)dx=\sum_{j=1}^{J}\omega_jg(\xi_j)-\sum_mY_J(p_m)Res(g)_{p_m}+R_J(g).
\end{equation}
Here $\omega_j$ and $\xi_j$ denote respectively the Gaussian quadrature weights and nodes for $j=1,2,...,J$, $Res(g)_{p_m}$ is the residue of $g$ at the pole $p_m$, and from \cite{gautschi1981survey} we have 
\begin{equation}
	R_J(g)=\frac{1}{2\pi i}\int_{|z|=l}Y_J(z)g(z)dz
\end{equation}
with 
\begin{equation}
	Y_J(z):=\frac{1}{P_J}\int_{-1}^{1}\frac{P_J(x)}{z-x}dx,\qquad z\in\mathbb{C}\backslash[-1,1]
\end{equation}
and $P_J$ being the Legendre orthogonal polynomial of degree $J$.

\begin{remark} If $g(z)$ is analytic in $\overline{B(l)}$, the Gaussian quadrature \cref{respole} rewritten as follows:
\begin{equation}\label{Gaussquadrature}
	\int_{-1}^{1}g(x)dx=\sum_{j=1}^{J}\omega_jg(\xi_j)+R_J(g).
\end{equation}
\end{remark}

The following estimate of $R_J(g)$ is from \cite{baffet2019gauss}.
\begin{lemma}\label{Theorem3.1}
	There exists a positive integer $J_*$ and a positive constant $C$, independent of 
	$l$, such that
	\begin{equation}\label{Gaussestimation}
		|R_J(g)|\leq C(l+\sqrt{l^2-1})^{-2J}\max\limits_{|z|=l}|g(z)|, \qquad \forall J>J_* .
	\end{equation}
\end{lemma}

\subsubsection{SOE approximation of   $E_{\alpha}(-t^{\alpha})$}

Applying the variable substitution $x=s^{-\alpha}$ in the integral \eqref{identityMLF},  we  get
\begin{equation}\label{identityMLF1}
		\begin{aligned}
			E_{\alpha}(-t^{\alpha})&= \int_{0}^{\infty}f(x,t,\alpha)dx
		\end{aligned}
	\end{equation}
with
\begin{equation}
	f(x,t,\alpha):=\frac{\sin(\alpha\pi)}{\alpha\pi} \frac{e^{-tx^{-\frac 1\alpha} }}{x^{2}+2x\cos\alpha\pi+1}.
	\nonumber
\end{equation}

In order to derive an efficient SOE approximation for $E_{\alpha}(-t^{\alpha})$, we truncate the integral \cref{identityMLF1} to a finite interval, then subdivide the finite interval into a set of non-overlapping intervals: 
\begin{equation}\label{non-overlapping-intervals}
	\big(c_k-r_k \,,\, c_k+r_k\big),\ k=0,1,\cdots,K.
\end{equation} 
Specifically, we decompose \cref{identityMLF1} into a summation form 
\begin{equation}\label{approxiamationMLF}
	E_{\alpha}(-t^{\alpha})=\sum_{k=0}^{K}\int_{c_k-r_k}^{c_k+r_k}f(x,t,\alpha)dx+\int_{c_k+r_k}^{\infty}f(x,t,\alpha)dx.
\end{equation}
and apply the Gauss quadrature discussed in \cref{Gausssec} for the integral over each subinterval $\big(c_k-r_k \,,\, c_k+r_k\big)$. Following a similar strategy as outlined in \cite{yan2017fast,baffet2019gauss}, we determine the centers $c_0,\, c_1,\cdots,c_K$ and radii $r_0,\, r_1,\cdots,r_K$ for the intervals defined in \cref{non-overlapping-intervals} as follows: let $q>1$ be a constant,
\begin{equation}\label{ckrk}
	c_0=r_0=\frac{1}{2}, \quad
	c_k 
	=\frac{(q+1)q^{k-1}}{2}, \quad r_k=\frac{(q-1)q^{k-1}}{2},   \qquad k=1,2,\cdots,K.
\end{equation}
\begin{remark}
	Note that  the choice of an exponential function $q^k$  in \cref{ckrk}   leads to an exponential increase in subinterval length. 	
  This  then allows, as  shown in   \cref{numK}, as few  subintervals $\big(c_k-r_k \,,\, c_k+r_k\big)$ in   \cref{approxiamationMLF} as possible  while keeping the required approximation accuracy of the finite sum. In addition,
		  when applying the Gauss quadrature  formula to each term $\int_{c_k-r_k}^{c_k+r_k}f(x,t,\alpha)dx$,  as shown in \cref{Rerror}, the remaining term can be controlled uniformly.
\end{remark}

 We now at a position to  compute  the integral  term 
  $$ \int_{c_k-r_k}^{c_k+r_k}f(x,t,\alpha)dx$$ 
  for each $k$ so as to get the desired SOE approximation. To this end, 
we apply the integration variable substitution $x=r_ky+c_k$ to obtain
\begin{equation}\label{2.19}
	\begin{aligned}
		\int_{c_k-r_k}^{c_k+r_k}f(x,t,\alpha)dx&
		=\int_{-1}^{1}g_k(y,t)dy,\\
	\end{aligned}
\end{equation}
where
\begin{equation}\label{gk}
		g_k(y,t):=\frac{\sin(\alpha\pi)}{\alpha\pi} \frac{r_ke^{-t(r_ky+c_k)^{-\frac 1\alpha}}}{  (r_ky+c_k)^{2 }+2(r_ky+c_k)\cos\alpha\pi +1}, \quad k=0,1,  \cdots, K.
\end{equation}
Notice that $0<\alpha<1$ and 
$$z^{2 }+2z\cos(\alpha\pi) +1=\left(z + \cos\alpha\pi+i \sin\alpha\pi\right)\left(z+ \cos\alpha\pi-i \sin\alpha\pi\right),$$
then we easily know that for any $k$, $g_k(y,t)$ has two simple poles:
\begin{align}\label{2poles}
	\left\{\begin{aligned} & \zeta_{k,1} 
	=-\frac{c_k}{r_k}+ \frac{1}{r_k}\big(\cos (1-\alpha)\pi+i\sin(1-\alpha)\pi\big),\\
	& \zeta_{k,2}=\bar{\zeta}_{k,1}=-\frac{c_k}{r_k}+ \frac{1}{r_k}\big(\cos (1-\alpha)\pi-i\sin(1-\alpha)\pi\big).
	\end{aligned}
	\right.
\end{align}
Denote 
\begin{equation}\label{q1q2def}
		q_1:=\sqrt{  5- 4\cos(1-\alpha)\pi },\quad
		q_2:=\frac1 {q-1} \sqrt{  {(q+1)^2}- 4(q+1)\cos(1-\alpha)\pi+  4 }.
	\end{equation}
	For $0<\alpha <1$ we easily have 
	$$q_1 > \sqrt{  5- 4 } =1$$ and $$   q_2> \frac1 {q-1} \sqrt{  {(q+1)^2}- 4(q+1) +  4 }=1. $$
	Thus,   it is reasonable to make the following assumption on $l$ and $q$: 
	\begin{equation}\label{constraint-lq}
		1<l<\min\{1+\frac{2}{q}, q_1,q_2\}.
	\end{equation}

 \begin{remark}
 From \eqref{ckrk} 
 we easily know that, for $k=0,1,  \cdots, K,$ 
\begin{align}\label{zeta-k1k2-norm}
|\zeta_{k,1}|=|\zeta_{k,2}|&=\sqrt{\left(-\frac{q+1}{q-1}+ \frac{1}{r_k}\cos(1-\alpha)\pi\right)^2+ \frac{1}{r_k^2}\sin^2(1-\alpha)\pi}\nonumber\\
&=\sqrt{ \frac{(q+1)^2}{(q-1)^2}- \frac{4(q+1)}{q^{k-1}(q-1)^2}\cos(1-\alpha)\pi+ \frac{4}{q^{2(k-1)}(q-1)^2} }.
\end{align}   
This relation, together with   $0<\alpha<1$, $q>1$ and the assumption \eqref{constraint-lq}, 
further  implies that
 \begin{equation}\label{zeta_k1k2}
 |\zeta_{k,1}|=|\zeta_{k,2}| \left\{
 \begin{aligned}
 = &\ \sqrt{5-4\cos(1-\alpha)\pi}=q_1  >l & \text{if } k=0,\\
 =&\  \frac1 {q-1} \sqrt{  {(q+1)^2}- 4(q+1)\cos(1-\alpha)\pi+  4 } =q_2 >l & \text{if } k=1,\\
 \geq &\ \left|\frac{q+1}{q-1}- \frac{2}{q^{k-1}(q-1)}\right| \geq \frac{q+1}{q-1}- \frac{2}{q(q-1)}=1+\frac{2}{q}>l& \text{if } k\geq 2.\\ 
\end{aligned}
 \right.
  \end{equation}

\end{remark}

By  
 \eqref{zeta_k1k2} 
it is easy to see that  $g_k(\cdot,t)$     has no poles in the disk $B(l)$  for any $k $.  Thus, from the Gaussian quadrature formula  \cref{Gaussquadrature} 
%
we have 
\begin{equation}\label{approxgkk}
	\int_{-1}^{1}g_k(x,t)dx=\sum_{j=1}^{J}b_{kj}e^{-ta_{kj}}+R_J(g_k), \quad k=0,1,\cdots,  K,
\end{equation}
where,  for $k=0,1,\cdots,  K$ and $ j=1,\cdots,J,$
\begin{equation}\label{a-kj}
\left\{\begin{aligned}	&a_{kj}=(r_k\xi_j+c_k)^{-\frac{1}{\alpha}} , \\
 &b_{kj}=\frac{\sin(\alpha\pi)}{\alpha\pi}\cdot \frac{\omega_jr_k}{  (r_k\xi_j+c_k)^{2 }+2(r_k\xi_j+c_k)\cos\alpha\pi +1}.
\end{aligned}
\right.	
\end{equation}
We recall that $\omega_j$ and $\xi_j$ denote the Gaussian quadrature weights and nodes, respectively. 

Substituting \eqref{approxgkk} and \eqref{2.19} into \eqref{approxiamationMLF},   we finally get the   
sum-of-exponentials approximation of the Mittag-Leffler function 
	\begin{align}\label{SOEapp}
		E_{\alpha}(-t^\alpha)&=\sum_{k=0}^{K}\left(\sum_{j=1}^{J}b_{kj}e^{-ta_{kj} }+R_J(g_k)\right)+\int_{q^K}^{\infty}f(x,t,\alpha)dx\nonumber \\
		&= \sum_{k=0}^{K}\sum_{j=1}^{J}b_{kj}e^{-ta_{kj} }+R_{soe}(t), 
	\end{align}
where 
	\begin{align}
		&R_{soe}(t):=\sum_{k=0}^{K}R_J(g_k)+\int_{q^K}^{\infty}f(x,t,\alpha)dx. \label{remain}
	\end{align}
In what follows we shall estimate the remaining term  $R_{soe}(t)$. For the truncation integral  term of \eqref{remain}, we easily obtain the following conclusion:
\begin{lemma}\label{numK}
	For  $0<\alpha<1, q>1$ and $t> 0$, there holds
	\begin{equation}
		\left|\int_{q^K}^{\infty}f(x,t,\alpha)dx\right|\leq\frac{1}{q^{  K}-1}.
	\end{equation}
\end{lemma}

\begin{proof}
	Notice that 
	\begin{equation}
		\begin{aligned}
			&\int_{q^K}^{\infty}f(x,t,\alpha)dx=\frac{\sin\alpha\pi}{\alpha\pi}\cdot\int_{q^K}^{\infty}\frac{e^{-tx^{-\frac{1}{\alpha}}}}{x^2+2x\cos \alpha\pi+1}dx.
		\end{aligned}
	\end{equation}
	For $0<\alpha<1$ and $t$, $x>0$, we have 
	\begin{equation}
		0<e^{-tx^{-\frac{1}{\alpha}}}\leq1, \quad 0<\frac{\sin\alpha\pi}{\alpha\pi}<1,
		\nonumber
	\end{equation}
	  and then
	\begin{equation}
		\begin{aligned}
			\left|\int_{q^K}^{\infty}f(x,t,\alpha)dx\right|
			&\leq \left|\int_{q^K}^{\infty}\frac{1}{x^2+2x\cos\alpha\pi+1}dx\right| \\
			&\leq \int_{q^K}^{\infty}\frac{1}{x^2-2x+1}dx=\frac{1}{q^{ K}-1}.
		\end{aligned}
		\nonumber
	\end{equation}
	This finishes the proof.
\end{proof}
%


For the term $R_{J}(g_k)$ in \eqref{remain}, we have the following result:

%

\begin{lemma}\label{Rerror}
	For $0<t\leq T$, $0<\alpha<1 $, $q>1$ and $l$ satisfying \eqref{constraint-lq}, there holds
	\begin{align}\label{paraC}
		|R_{J}(g_k)|\leq C_{\alpha,T,q}(l+\sqrt{l^2-1})^{-2J}, \quad k=0, 1, \cdots,K,
	\end{align}
	where
	\begin{equation*}
		\begin{aligned}
			C_{\alpha,T,q}=\left\{\begin{aligned}
				& \frac{2qe^{T}}{(q-1)(q_1-l)^{2}}  &\text{if } k=0,\\
				&\frac{2qe^{T}}{(q-1)(q_2-l)^{2}}  &\text{if } k=1,\\
				&\frac{2qe^{T}}{(q-1)(1+\frac{q}{2}-l)^{2}}  &\text{if } k\geq2.\\
			\end{aligned}\right.
		\end{aligned}
	\end{equation*}
\end{lemma}
\begin{proof}
	In light of  \cref{Gaussestimation}, 
	for each $R_{J}(g_k)$ we only need to estimate the term $\max\limits_{|z|=l}|g_k(z,t)|$. By  
 \eqref{gk} and \eqref{2poles} we have 
	\begin{equation}\label{max}
	\begin{aligned}
		\max 
		\limits_{|z|=l}|g_k(z,t)|&< 
		\frac{\max\limits_{|z|=l}|e^{-t(r_kz+c_k)^{-\frac1\alpha}}|}{r_k \min\limits_{|z|=l}|z-\zeta_{k,1}|\ |z-\zeta_{k,2}|}\ <	\frac{\max\limits_{|z|=l}|e^{-tr_k^{-\frac1\alpha}(z+\frac{q+1}{q-1})^{-\frac1\alpha}}|}{r_k\min\limits_{|z|=l}|z-\zeta_{k,1}|\ |z-\zeta_{k,2}|}.
	\end{aligned}
\end{equation}
From \eqref{ckrk} and \cref{zeta_k1k2} we easily know that
\begin{align}
	 \frac 1 r_k 
	 \leq \max\{2, \frac{2}{q-1} \}< 2+ \frac{2}{q-1}=\frac{2q}{q-1},\quad  k=0,1, 2, \cdots
\end{align} 
and 
\begin{equation}
	\begin{aligned}
		&\min\limits_{|z|=l}|z-\zeta_{k,1}|\ |z-\zeta_{k,2}|\geq \left\{\begin{aligned}
			&(q_1-l)^{2},\quad &\text{if }k=0,\\
			&(q_2-l)^{2},\quad &\text{if }k=1,\\
			&(1+\frac{q}{2} -l)^{2},\quad &\text{if }k\geq2.\\
		\end{aligned}\right.
	\end{aligned}
\end{equation}


To estimate the term $\max\limits_{|z|=l}|e^{-tr_k^{-\frac1\alpha}(z+\frac{q+1}{q-1})^{-\frac1\alpha}}|$, 
 we assume   $z=l(\cos\theta+	i\sin\theta)$ with $-\pi<\theta\leq\pi$ and obtain 
$$ z+\frac{q+1}{q-1}  =l\cos\theta+\frac{q+1}{q-1}+i l\sin\theta=\left[\left(l\cos\theta+\frac{q+1}{q-1}\right)^2+l^2\sin^2\theta\right]^{\frac12}\ (
\cos\tilde\theta+	i\sin\tilde\theta)$$
with
$\tilde{\theta}=\arctan 
\frac{l\sin\theta}{l\cos\theta+\frac{q+1}{q-1}}.
$ This means 
$$ \left(z+\frac{q+1}{q-1}\right)^{-\frac1\alpha} =\left[\left(l\cos\theta+\frac{q+1}{q-1}\right)^2+l^2\sin^2\theta\right]^{-\frac{1}{2\alpha}}\left(\cos(-\frac{\tilde\theta}{\alpha})+	i\sin(-\frac{\tilde\theta}{\alpha})\right) $$
and 
\begin{align*}
-Re \left(z+\frac{q+1}{q-1}\right)^{-\frac1\alpha} & =-\left[\left(l\cos\theta+\frac{q+1}{q-1}\right)^2+l^2\sin^2\theta\right]^{-\frac{1}{2\alpha}}\cos(-\frac{\tilde\theta}{\alpha})\\
 &\leq \left[\left(l\cos\theta+\frac{q+1}{q-1}\right)^2+l^2\sin^2\theta\right]^{-\frac{1}{2\alpha}}\leq \left(l+\frac{q+1}{q-1}\right)^{-\frac{1}{\alpha}}.
 \end{align*}
Thus,   we have
	\begin{align}\label{inequality2}
			\max\limits_{|z|=l}|e^{-tr_k^{-\frac1\alpha}(z+\frac{q+1}{q-1})^{-\frac1\alpha}}|
			&=\max\limits_{|z|=l}\left|e^{-tr_k^{-\frac{1}{\alpha}}Re(z+\frac{c_k}{r_k})^{-\frac1\alpha}}\right|\nonumber\\
			&\leq e^{T (\frac{2q}{q-1})^{\frac{1}{\alpha}}
			\left(\frac{q+1}{q-1}+l\right)^{-\frac{1}{\alpha}}}\nonumber\\	
			&\leq e^{T (\frac{2q}{ q+1+l(q-1)})^{\frac{1}{\alpha}}}\leq e^T.
		\end{align}
Finally, combining   \cref{Gaussestimation} and  the inequalities \cref{max}-\cref{inequality2} gives the desired estimate \cref{paraC}.
%
%
%
\end{proof}

 Recalling \eqref{a-kj},   we give a compact form of the SOE approximation \cref{SOEapp} as follows:
  \begin{align}\label{SOEappnew}
		E_{\alpha}(-t^\alpha)
		=\sum_{j=1}^{N_{exp}}b_je^{-a_jt}+R_{soe}(t),
	\end{align}
where
    $N_{exp}=(K+1)J$,  and   $a_j$ and $b_j$ are the j-th elements of 
\begin{align*}
 &[a_{01},a_{02},...,a_{0J},a_{11},a_{12},...,a_{1J},...,a_{(K+1)1},...,a_{(K+1)(J-1)},a_{(K+1)J}]
\end{align*}
and 
\begin{align*}
 &[b_{01},b_{02},...,b_{0J},b_{11},b_{12},...,b_{1J},...,b_{(K+1)1},...,b_{(K+1)(J-1)},b_{(K+1)J}],
\end{align*}
 respectively.   
 
 We are now at a position to estimate the  the SOE approximation error
\begin{equation}\label{approxerrorE}
R_{soe}(t) = E_{\alpha}(-t^\alpha)-\sum_{j=1}^{N_{exp}}b_je^{-a_jt}, \quad 0<t\leq T.
\end{equation}
 In light of Lemmas \ref{numK} and \ref{Rerror} and the relation \eqref{remain},   we immediately get the following  main conclusion: 
\begin{theorem}\label{approxerror}
	For   $0<\alpha<1$, $q>1$ and  $1<l<\min\{1+\frac{2}{q}, q_1,q_2\}$, 
	there holds 
	\begin{equation}\label{approxerror11}
	\left|R_{soe}(t)\right|	 \leq C_{\alpha,T,q}(K+1)(l+\sqrt{l^2-1})^{-2J}+\frac{1}{q^{  K}-1},\quad 0<t\leq T.
	\end{equation}
Moreover, for any $0<\epsilon<1$ there holds 
\begin{align}\label{SOEapp2}
		\left|R_{soe}(t)\right|=\mathcal{O}(\epsilon),\quad 0<t\leq T,
	\end{align}
  provided that
\begin{equation}\label{KJ-order}
	K=\mathcal{O}(|\log\epsilon|),\qquad J=\mathcal{O}(|\log(\epsilon^{-1}|\log\epsilon|)|).
\end{equation}

\end{theorem}

\begin{remark}\label{rmk2.4} \cref{approxerror} means that for a given tolerance error $\epsilon$, the computation  complexity of the SOE approximation \cref{SOEappnew} is
$$ N_{exp}=(K+1)J= \mathcal{O}(|\log\epsilon|^2).$$
Furthermore, 	denote $N :=\frac{T}{\triangle t}$ with $\triangle t<1$ being the temporal step size, then for $\epsilon=\triangle t$  we have
	\begin{equation}
		N_{exp}  =\mathcal{O}(\log^2N).
	\end{equation}
\end{remark}
 
\begin{remark}\label{selectJK}
In view of  \eqref{approxerror11} and \eqref{KJ-order}, we shall select 
	\begin{align}\label{KJ-choice}
		K=\left\lceil\frac{|\log\varepsilon|}{\log q}\right\rceil,\qquad J=\left\lceil\frac{\log\left(\varepsilon^{-1}|\log\varepsilon|\right)}{2\log q\log l}\right\rceil
	\end{align}
 in   the numerical implementation  (cf.  Section 4),	where \( \lceil \cdot \rceil \) denotes the ceiling function, which rounds up to the nearest integer.
\end{remark}

\section{Numerical schemes for the fractional viscoelastic model}

In this section, we present two fully discrete mixed finite element schemes for the fractional viscoelastic model \cref{model}. One is based on the weak form \cref{weakformoftraditon} and applies the traditional L1 scheme and the Newmark scheme to discretize the  time-fractional derivative and the  second  time  derivative, respectively. The other one is based on the weak form \cref{weak-form} and adopts the SOE approximation for the Mittag-Leffler function. 

Let  $\uuline{\mathrm{H}_h}\subset\uuline{\mathrm{H}}(\mathrm{\textbf{div}},\Omega,S)$ and $\underline{\mathrm{V}_h}\subset\uline{L}^2(\Omega)$ be two finite-dimensional spaces for stress and displacement approximations, respectively.

For any positive integer $N$, let
$$
\{t_{n}:\ t_{n} = n\triangle t,\ 0\leq n\leq N \}
$$
be a uniform partition of the time interval $(0,T]$ with the time step size $\triangle t = T/N$.

\subsection{L1-Newmark mixed finite element scheme}

In view  of the weak form  \cref{weakformoftraditon}, 
the  generic semi-discrete mixed conforming finite element scheme  for the fractional viscoelastic model \cref{model} reads: 

Find $ \sigma_h(t)\in \uuline{\mathrm{H}_h}$ and $\textbf{u}_h(t)\in \underline{\mathrm{V}_h}$ such that 
\begin{equation}\label{semimodel}
	\left\{
	\begin{array}{ll}
		\langle\rho \textbf{u}_{h,tt} \,,\, \textbf{v}_h\rangle=\langle\mathrm{\textbf{div}}\sigma_h \,,\, \textbf{v}_h\rangle+\langle\textbf{f} \,,\, \textbf{v}_h\rangle,& \forall \, \textbf{v}_h\in \underline{\mathrm{V}_h},\\
		\langle\mathbb{D}^{-1}\sigma_h \,,\, \chi_h\rangle+\tau_{\sigma}^{\alpha}\langle\frac{\partial^{\alpha}\mathbb{D}^{-1}\sigma_h}{\partial t^{\alpha}} \,,\, \chi_h\rangle+\langle\mathrm{\textbf{div}}\chi_h \,,\, \textbf{u}_h\rangle+\tau_{\varepsilon}^{\alpha}\langle\mathrm{\textbf{div}}\chi_h \,,\, \frac{\partial^{\alpha}\textbf{u}_h}{\partial t^{\alpha}}\rangle=0, & \forall \, \chi_h \in \uuline{\mathrm{H}_h},\\
		\textbf{u}_h(0)=I_{\mathbb{V}_h}\textbf{u}_0,\ \textbf{u}_{h,t}(0)=I_{\mathbb{V}_h}\textbf{v}_0,\  \sigma_h(0)=I_{\mathbb{H}_h}\sigma_0,
		
	\end{array}
	\right.
\end{equation}
where $I_{\mathbb{V}_h}$ and $I_{\mathbb{H}_h}$ denote the projection operators onto $\underline{\mathrm{V}_h}$ and $\uuline{\mathrm{H}_h}$, respectively.

Let $\{\varphi_i\}_{i=1}^{r}$ and $\{\kappa_i\}_{i=1}^{s}$ be bases of $\uuline{\mathrm{H}_h}$ and $\uline{\mathrm{V}_h}$, respectively, and introduce matrices $\textbf{A}=(\textbf{A}_{ij})_{r\times r}$, $\textbf{B}=(\textbf{B}_{ij})_{r\times s}$, $\textbf{C}=(\textbf{C}_{ij})_{s\times s}$ with
\begin{align*}
	\textbf{A}_{ij}=\langle\mathbb{D}^{-1}\varphi_i \,,\, \varphi_j\rangle,\quad \textbf{B}_{ij}=\langle\mathrm{\textbf{div}}\varphi_i \,,\, \kappa_j\rangle,\quad \textbf{C}_{ij}=\langle\rho\kappa_i \,,\, \kappa_j\rangle.
\end{align*}
We write $\sigma_h=\sum_{i=1}^{r}\beta_i(t)\varphi_i,\ \textbf{u}_h=\sum_{j=1}^{s} U_j(t)\kappa_j,\ \eta_j=\langle\textbf{f}(t) \,,\, \kappa_j\rangle$, and denote
\begin{align*}
	\beta(t):=(\beta_1,\beta_2,\cdots,\beta_r)^{\mathrm{T}},\ U(t):=(U_1,U_2,\cdots,U_s)^{\mathrm{T}},\
	\eta(t):=(\eta_1,\eta_2,\cdots,\eta_s)^{\mathrm{T}}.
\end{align*}
Then we can rewrite \cref{semimodel} as the following matrix form:
\begin{equation}\label{semimatrix}
	\left\{
	\begin{array}{ll}
		\textbf{C}U_{tt}-\textbf{B}^{\mathrm{T}}\beta=\eta,\\
		\textbf{A}\beta+\tau_{\sigma}^{\alpha}\textbf{A}\frac{\partial^\alpha\beta}{\partial t^\alpha}+\textbf{B}U+\tau_{\varepsilon}^{\alpha}\textbf{B}\frac{\partial^\alpha U}{\partial t^\alpha}=0,\\
	\end{array}
	\right.
\end{equation}
with the initial data $U(0)=I_{\mathbb{V}_h}\textbf{u}_0,\ U_t(0)=I_{\mathbb{V}_h}\textbf{v}_0,\ \beta(0)=I_{\mathbb{H}_h}\sigma_0$.

To discretize the term $U_{tt}$ in \eqref{semimatrix}, we choose the Newmark scheme \cite{newmark1959method} as follows: 
\begin{equation}\label{Newmark}
\left\{
\begin{aligned}
&U_{tt}(t_n)=\frac{1}{{\triangle t}^2\theta_2}\left(U(t_n)-U(t_{n-1})-\triangle tU_t(t_{n-1})-\frac{{\triangle t}^2}{2}(1-2\theta_2)U_{tt}(t_{n-1})\right),\\
&U_t(t_{n})=U_t(t_{n-1})+\triangle t\left[(1-\theta_1)U_{tt}(t_{n-1})+\theta_1U_{tt}(t_{n})\right],
		\end{aligned}
		\right.
	\end{equation}
where the choice of parameters $(\theta_1,\theta_2)$ depends on the requirement of accuracy and stability for the scheme (cf. Remark \ref{rmk3.1}). In  our numerical experiments in next section we choose  $\theta_1=\frac{1}{2}$ and $\theta_2=\frac{1}{4}$.
\begin{remark}\label{rmk3.1}
	We list four well-known members of the Newmark method \cite{1983Computational,newmark1959method}: 
	\begin{table}[H]
		\setlength{\tabcolsep}{3mm}
		\renewcommand{\arraystretch}{1.3}
		\begin{tabular}{|c|cc|c|}  
			\hline
		Four   methods	  &$\theta_1$&$\theta_2$&Accuracy\\ \hline
			Newmark explicit method&$\frac{1}{2}$&0&second order\\ \hline
			Fox-Goodwin method&$\frac{1}{2}$&$\frac{1}{12}$&third order\\ \hline
			Linear average acceleration method&$\frac{1}{2}$&$\frac{1}{6}$&second order\\ \hline
			Constant average acceleration method&$\frac{1}{2}$&$\frac{1}{4}$&second order\\ \hline
		\end{tabular}
	\end{table} 
	\noindent  We note that the constant average acceleration Newmark method $(\theta_1=\frac{1}{2}$, $\theta_2=\frac{1}{4})$  is second-order accurate and unconditionally stable.
\end{remark}
For the discretization of Caputo fractional derivative $\frac{\partial^\alpha U}{\partial t^\alpha}$, the following L1 scheme  
is commonly used:
\begin{equation}\label{L1}
	\frac{\partial^\alpha U}{\partial t^\alpha}(t_n)=\frac{{\triangle t}^{-\alpha}}{\Gamma(2-\alpha)}\left[a_0^{\alpha}U(t_n)-\sum_{k=1}^{n-1}\left(a_{n-k-1}^{\alpha}-a_{n-k}^{\alpha}\right)U(t_k)-a_{n-1}^{\alpha}U(0)\right]
\end{equation}
where $a_k^{\alpha}=(k+1)^{1-\alpha}-k^{1-\alpha}$ for $k=0,1, \cdots, n-1$.

Substituting the L1 scheme \eqref{L1} and the Newmark scheme \eqref{Newmark} into \eqref{semimatrix} leads to the following fully discrete linear system: for $n=1,\ 2,\cdots,N$
	\begin{align}\label{fully1}
		&\left(\frac{\textbf{C}}{{\triangle t}^2\theta_2}+\frac{1+L_{\varepsilon}}{1+L_\sigma}\textbf{B}^{\mathrm{T}}\textbf{A}^{-1}\textbf{B}\right)U(t_n) \nonumber\\
		=\ &\eta(t_n)+\frac{L_\sigma}{1+L_\sigma}\textbf{B}^{\mathrm{T}}K_{\sigma,n-1}+\frac{L_\varepsilon}{1+L_\sigma}\textbf{B}^{\mathrm{T}}\textbf{A}^{-1}\textbf{B}K_{u,n-1}\nonumber\\
		&\quad +\frac{\textbf{C}}{{\triangle t}^2\theta_2}\left(U(t_{n-1})+\triangle tU_t(t_{n-1})+\frac{{\triangle t}^2}{2}(1-2\theta_2)U_{tt}(t_{n-1})\right) , 
	\end{align}
where 
\begin{equation*}
	\begin{aligned}
		&L_\sigma:=\frac{\tau_{\sigma}^\alpha}{\triangle t^\alpha\Gamma(2-\alpha)}, \quad &K_{\sigma,n-1}:=\sum_{k=1}^{n-1}\left(a_{n-k-1}^{\alpha}-a_{n-k}^{\alpha}\right)\beta(t_k)+a_{n-1}^{\alpha}\beta(0), \\
		&L_\varepsilon:=\frac{\tau_{\varepsilon}^\alpha}{\triangle t^\alpha\Gamma(2-\alpha)}, \quad &K_{u,n-1}:=\sum_{k=1}^{n-1}\left(a_{n-k-1}^{\alpha}-a_{n-k}^{\alpha}\right)U(t_k)+a_{n-1}^{\alpha}U(0).
	\end{aligned}
\end{equation*}
Define
\begin{align*}
	H_{n-1}:=\frac{L_\varepsilon}{1+L_\sigma}\textbf{B}^{\mathrm{T}}\textbf{A}^{-1}\textbf{B}K_{u,n-1}+\frac{L_\sigma}{1+L_\sigma}\textbf{B}^{\mathrm{T}}K_{\sigma,n-1},
\end{align*}
we easily have the following recurrence relation:
\begin{equation}\label{His}
	\begin{aligned}
		H_{n-1}=&\frac{L_\varepsilon}{1+L_\sigma}\textbf{B}^{\mathrm{T}}\textbf{A}^{-1}\textbf{B}K_{u,n-1}+\frac{L_\sigma}{1+L_\sigma}\textbf{B}^{\mathrm{T}}\left[\sum_{k=1}^{n-1}\left(a_{n-k-1}^{\alpha}-a_{n-k}^{\alpha}\right)\beta(t_k)+a_{n-1}^{\alpha}\beta(0)\right] \\
		=&\frac{L_\varepsilon}{1+L_\sigma}\textbf{B}^{\mathrm{T}}\textbf{A}^{-1}\textbf{B}K_{u,n-1}+\frac{L_\sigma}{1+L_\sigma}\sum_{k=1}^{n-1}\left(a_{n-k-1}^{\alpha}-a_{n-k}^{\alpha}\right)H_{k-1} \\
		&-\frac{L_\sigma(1+L_{\varepsilon})}{(1+L_\sigma)^2}\textbf{B}^{\mathrm{T}}\textbf{A}^{-1}\textbf{B}\left[\sum_{k=1}^{n-1}\left(a_{n-k-1}^{\alpha}-a_{n-k}^{\alpha}\right)U(t_k)+a_{n-1}^{\alpha}U(0)\right] \\
		=&\frac{L_\varepsilon-L_\sigma}{(1+L_\sigma)^2}\textbf{B}^{\mathrm{T}}\textbf{A}^{-1}\textbf{B}K_{u,n-1}+\frac{L_\sigma}{1+L_\sigma}\sum_{k=1}^{n-1}\left(a_{n-k-1}^{\alpha}-a_{n-k}^{\alpha}\right)H_{k-1}. \\
	\end{aligned}
\end{equation}
In conclusion, we have the following L1-Newmark mixed finite element
algorithm: 
\begin{algorithm}[H]
	\caption{L1-Newmark MFE scheme }\label{algo1}
	\begin{algorithmic}[1]
		\Require $U(0),\ U_t(0),\ \eta(0),\ H_0,\ U_{tt}(0)=\textbf{C}^{-1}\left(\eta(0)-\frac{1+L_{\varepsilon}}{1+L_\sigma}\textbf{B}^{\mathrm{T}}\textbf{A}^{-1}\textbf{B}U(0)\right)$.
		\Ensure $U(t_N)$
		\For{$n\gets 1, N$}
		\State Solve $U(t_n)$ with the   scheme 
		\begin{equation}\nonumber
			\begin{aligned}
				&\left(\frac{\textbf{C}}{{\triangle t}^2\theta_2}+\frac{1+L_{\varepsilon}}{1+L_\sigma}\textbf{B}^{\mathrm{T}}\textbf{A}^{-1}\textbf{B}\right)U(t_n)\\
				=&\eta(t_n)+H_{n-1}+ \frac{\textbf{C}}{{\triangle t}^2\theta_2}\left(U(t_{n-1})+\triangle tU_t(t_{n-1})+\frac{{\triangle t}^2}{2}(1-2\theta_2)U_{tt}(t_{n-1})\right).
			\end{aligned}
		\end{equation}
	    \State Calculate and store history variable $H_n$. 
	    \State Compute $U_t(t_n)$ and $U_{tt}(t_n)$ through \cref{Newmark}.
	    \EndFor 
	    \State Return $U(t_N)$.
	\end{algorithmic}
\end{algorithm}

Note that at each time step  we need to calculate and store the history variable $H_n$. This means that  \cref{algo1} requires $\mathcal O(N_sN)$ memory complexity and $O(N_sN^2)$ computation complexity. Here we simply denote by $\mathcal O(N_s)$  the complexities of memory and computation related to the spatial discretization.
As $N$ is large,  the complexities of memory and computation of  \cref{algo1} 
may   create obstacles for a long time simulation. Therefore, in the following subsection  we shall provide a fast numerical scheme  based on the weak form \cref{weak-form}. 

\subsection{Fast numerical scheme with SOE approximation}

In view  of the weak form  \cref{weak-form}, 
we have the following semi-discrete mixed conforming finite element scheme  for the fractional viscoelastic model \cref{model}: 

 Find $ \sigma_h(t)\in \uuline{\mathrm{H}_h}$ and $\textbf{u}_h(t)\in \underline{\mathrm{V}_h}$ such that 
  \begin{equation}\label{semimodelofnewconstitutive}
	\left\{
	\begin{array}{ll}
		\langle\rho \textbf{u}_{h,tt} \,,\, \textbf{v}_h\rangle=\langle\mathrm{\textbf{div}}\sigma_h \,,\, \textbf{v}_h\rangle+\langle\textbf{f} \,,\, \textbf{v}_h\rangle,\quad \forall \, \textbf{v}_h\in \underline{\mathrm{V}_h},\\
		\langle\mathbb{D}^{-1}\sigma_h \,,\, \chi_h\rangle+\big((\frac{\tau_{\varepsilon}}{\tau_{\sigma}})^\alpha-1\big)\int_{0}^{t}E_{\alpha}(-(\frac{t-\tau}{\tau_{\sigma}})^\alpha)\langle\mathrm{\textbf{div}}\chi_h \,,\, \textbf{u}_{h,t}(\tau)\rangle d\tau+\langle\mathrm{\textbf{div}}\chi_h \,,\, \textbf{u}_{h,t}\rangle \\
		\quad =E_{\alpha}(-(\frac{t}{\tau_{\sigma}})^\alpha)\langle\mathbb{D}^{-1}\sigma_0 \,,\, \chi_h\rangle+E_{\alpha}(-(\frac{t}{\tau_{\sigma}})^\alpha)\langle\mathrm{\textbf{div}}\chi_h \,,\, \textbf{u}_0\rangle, \quad \forall \, \chi_h \in \uuline{\mathrm{H}_h},\\
		\textbf{u}_h(0)=I_{\mathbb{V}_h}\textbf{u}_0,\ \textbf{u}_{h,t}(0)=I_{\mathbb{V}_h}\textbf{v}_0,
	\end{array}
	\right.
\end{equation}
Using the same notations as in Section 3.1, we rewrite this system  as the following matrix form:
\begin{equation}\label{fastsemimatrix}
	\left\{
	\begin{array}{ll}
		\textbf{C}U_{tt}-\textbf{B}^{\mathrm{T}}\beta=\eta,\\
		\textbf{A}\beta+\big((\frac{\tau_{\varepsilon}}{\tau_{\sigma}})^\alpha-1\big)\textbf{B}\int_{0}^{t}E_{\alpha}(-(\frac{t-\tau}{\tau_{\sigma}})^\alpha)U_t(\tau)d\tau+\textbf{B}U_t=\iota,
	\end{array}
	\right.
\end{equation}
where $\iota=(\textbf{A}\beta(0)+\textbf{B}U(0))E_{\alpha}(-(\frac{t}{\tau_{\sigma}})^\alpha)$, with the initial data $U(0)=I_{\mathbb{V}_h}\textbf{u}_0,\ U_t(0)=I_{\mathbb{V}_h}\textbf{v}_0,\ \beta(0)=I_{\mathbb{H}_h}\sigma_0$.

According to the SOE approximation \cref{SOEapp2} 
in \cref{approxerror}, we have
	\begin{align*}
		E_{\alpha}(-(\frac{t-\tau}{\tau_{\sigma}})^\alpha) 
		=\sum_{j=1}^{N_{exp}}b_je^{-a_j(\frac{t-\tau}{\tau_{\sigma}})}+\mathcal{O}(\epsilon),
	\end{align*}
 which, together with  integration by parts, gives
	\begin{align}\label{Soe33}
		\int_{0}^{t}E_{\alpha}(-(\frac{t-\tau}{\tau_{\sigma}})^\alpha)U_t(\tau)d\tau=&\sum_{j=1}^{N_{exp}}b_j\left(U(t)-U(0)e^{-a_j\frac{t}{\tau_{\sigma}}}-\frac{a_j}{\tau_{\sigma}}\int_{0}^{t}e^{-a_j(\frac{t-\tau}{\tau_{\sigma}})}U(\tau)d\tau\right)\nonumber\\
		&+\mathcal{O}(\epsilon).
	\end{align}
Introduce the history variable 
\begin{align*}
	G_j(t):=\int_{0}^{t}e^{-a_j(\frac{t-\tau}{\tau_{\sigma}})}U(\tau)d\tau,
\end{align*}
and we have the following simple recurrence relation at $t=t_n$:
\begin{align}\label{Gin}
	G_j(t_n)&=e^{-\frac{a_j}{\tau_{\sigma}}\triangle t}G_j(t_{n-1})+\int_{t_{n-1}}^{t_n}e^{-a_j(\frac{t_n-\tau}{\tau_{\sigma}})}U(\tau)d\tau \nonumber\\
	&\approx e^{-\frac{a_j}{\tau_{\sigma}}\triangle t}G_j(t_{n-1})+T_{1,j}U(t_{n-1})+T_{2,j}U_t(t_{n-1})+T_{3,j}U_{tt}(t_{n-1}),
\end{align}
where
\begin{equation*}
	\begin{aligned}
		&T_{1,j}=\frac{\tau_{\sigma}}{a_j}\left(1-e^{-\frac{a_j}{\tau_{\sigma}}\triangle t}\right), \quad T_{2,j}=\frac{\tau_{\sigma}}{a_j}\triangle te^{-\frac{a_j}{\tau_{\sigma}}\triangle t}+\left(\triangle t-\frac{\tau_{\sigma}}{a_j}\right)T_{1,j}, \\ 
		&T_{3,j}=\left(\frac{\triangle t^2}{2}-\triangle t\frac{\tau_{\sigma}}{a_j}+\frac{\tau_{\sigma}^2}{a_j^2}\right)T_{1,j}+\frac{\tau_{\sigma}}{a_j}\triangle te^{-\frac{a_j}{\tau_{\sigma}}\triangle t}\left(\frac{\triangle t}{2}-\frac{\tau_{\sigma}}{a_j}\right). \\
	\end{aligned}
\end{equation*}
Finally,  
we apply the Newmark scheme \eqref{Newmark} to the semi-discrete scheme \eqref{fastsemimatrix} and use  \eqref{Gin} and \eqref{Soe33}  to obtain the   linear system
	\begin{align}\label{fast}
		&\left[\frac{\textbf{C}+\theta_1\triangle t \textbf{B}^{\mathrm{T}}\textbf{A}^{-1}\textbf{B}}{\triangle t^2\theta_2}+((\frac{\tau_{\varepsilon}}{\tau_{\sigma}})^\alpha-1)\textbf{B}^{\mathrm{T}}\textbf{A}^{-1}\textbf{B}\right]U(t_n) \nonumber\\
		=&\eta(t_n)+E(t_n)+Q_1U(t_{n-1})+Q_2U_t(t_{n-1})+Q_3U_{tt}(t_{n-1})\nonumber\\
		&+\left((\frac{\tau_{\varepsilon}}{\tau_{\sigma}})^\alpha-1\right)\textbf{B}^{\mathrm{T}}\textbf{A}^{-1}\textbf{B}\sum_{j=1}^{N_{exp}}\frac{a_jb_j}{\tau_{\sigma}}G_j(t_{n}),
	\end{align}
where
\begin{equation*}\label{Wis}
	\begin{aligned}
		&E(t_n)=(\sum_{j=1}^{N_{exp}}b_je^{-a_j\frac{t_n}{\tau_{\sigma}}})\left(\textbf{B}^{\mathrm{T}}\beta(0)+(\frac{\tau_{\varepsilon}}{\tau_{\sigma}})^\alpha \textbf{B}^{\mathrm{T}}\textbf{A}^{-1}\textbf{B}U(0)\right),\\
		&Q_1=\frac{\textbf{C}+\theta_1\triangle t \textbf{B}^{\mathrm{T}}\textbf{A}^{-1}\textbf{B}}{\triangle t^2\theta_2},\quad Q_2=\frac{\textbf{C}+(\theta_1-\theta_2)\triangle t \textbf{B}^{\mathrm{T}}\textbf{A}^{-1}\textbf{B}}{\triangle t\theta_2},\\ &Q_3=\frac{(1-2\theta_2)\textbf{C}+\triangle t(\theta_1-2\theta_2)\textbf{B}^{\mathrm{T}}\textbf{A}^{-1}\textbf{B}}{2\theta_2},
	\end{aligned}
\end{equation*}
and the history variable $G_j(t_{n})$ is computed by using the approximation formula \eqref{Gin}, i.e.
\begin{align}\label{Ginnew}
	G_j(t_n)= e^{-\frac{a_j}{\tau_{\sigma}}\triangle t}G_j(t_{n-1})+T_{1,j}U(t_{n-1})+T_{2,j}U_t(t_{n-1})+T_{3,j}U_{tt}(t_{n-1}).
\end{align}
In particular, $G_j(0)=0.$ The resulting fast algorithm, i.e.  \cref{algo2}, is given as follows:

\begin{algorithm}[H]
	\caption{Fast scheme}\label{algo2}
	\begin{algorithmic}[1]
		\Require $U(0),\ U_t(0),\ G_j(0)=0,\ U_{tt}(0)=\textbf{C}^{-1}\left(\eta(0)+\textbf{B}^{\mathrm{T}}\beta(0)\right) $ 
		
		\Ensure $U(t_N)$
		\State Calculate $N_{exp},\ a_j,\ b_j,\ T_{i,j},\ Q_i, \ j=1,...,N_{exp},\ i=1, 2, 3. $
		\For{$n\gets 1, N$}
			\State Calculate by \eqref{Ginnew} and store the history variable $G_j(t_n), \ j=1,...,N_{exp}$.

\State Solve $U(t_n)$ with  the scheme \eqref{fast}.
			\State Get  $U_{tt}(t_n)$ and $U_t(t_n)$ through \cref{Newmark}.
		\EndFor 
		\State Return $U(t_N)$.
	\end{algorithmic}
\end{algorithm}

Comparing with the   L1-Newmark algorithm (\cref{algo1}), we easily see that,  due to  $N>>N_{exp}$ (cf. \cref{rmk2.4}), \cref{algo2} reduces  the costs of memory and computation from $O(N_sN)$ and $O(N_sN^2)$ to $O(N_{exp}N_s)$ and to $O(N_sN_{exp}N)$, respectively.

\section{Numerical results}
In this section, we provide some numerical results to verify the efficiency of both the SOE  approximation \eqref{SOEappnew} (or \eqref{SOEapp}) and the fast scheme (\cref{algo2}). All the algorithms are implemented by using MATLAB 2023a and executed on a PC equipped with a 3.40 GHz processor, 32 GB of RAM, and running Windows 10.

\begin{example}[Test of SOE approximation accuracy]
	In this example, we evaluate the SOE approximation \eqref{SOEappnew} for the Mittag-Leffler function \( E_{\alpha}(-t^{\alpha}) \) under two distinct scenarios:
	\begin{enumerate}
		\item[1)]   Varying the parameters \( l \) and \( q \) while keeping the fractional order \( \alpha \) and the tolerance error $\varepsilon$ fixed;
		\item[2)] Varying the tolerance error $\varepsilon$ while keeping \( \alpha \), \( l \), and \( q \) fixed.
	\end{enumerate}

According to \cref{approxerror},   the parameters \( q \) and \( l \) are required to satisfy 
\begin{equation}\label{ql-choice}
q>1,   \quad 1<l<\min\{1+\frac{2}{q}, q_1,q_2\},
\end{equation} 
with
	\(q_1 = \sqrt{5 - 4\cos((1-\alpha)\pi)}, \)
	\(q_2 = \frac{1}{q-1} \sqrt{(q+1)^2 - 4(q+1)\cos((1-\alpha)\pi) + 4}.\)
The values of $q_3:=\min\left\{1 + \frac{2}{q}, q_1, q_2\right\}$   with $\alpha=0.2, \ 0.5, \ 0.7$ and  $q=2, \ 8, \ 9, \ 10, \ 11$    are listed in \cref{rangeofl}, based on which we compute  the following cases:
  $(q=2, l=1.5), (q=8, l=1.1), (q=9, l=1.1), (q=10, l=1.1), (q=11, l=1.09 )$.
 \begin{table}[ht]
	\centering
	\caption{The values of $q_3$ with different $q$ and $\alpha$: $1<l<q_3$.}
	\label{rangeofl}
	\begin{tabular}{cccccc}
		\toprule
		 $q$ & 2 & 8 & 9 & 10 & 11 \\
		\midrule
		$\alpha=0.2$ & 2 & 1.25 & 1.2222 & 1.2 & 1.1818 \\
		$\alpha=0.5$ & 2 & 1.25 & 1.2222 & 1.2 & 1.1818 \\
		$\alpha=0.7$ & 1.6275 & 1.1414 & 1.1214 & 1.1063 & 1.0945 \\
		\bottomrule
	\end{tabular}
\end{table}

As mentioned in  \cref{selectJK}, 	for given  $\varepsilon$, $l$ and $q$ the number $N_{exp}=(K+1)J$   of the SOE approximation is  determined  by \eqref{KJ-choice}, i.e.
\begin{align*}
	K=\left\lceil\frac{|\log\varepsilon|}{\log q}\right\rceil,\qquad J=\left\lceil\frac{\log\left(\varepsilon^{-1}|\log\varepsilon|\right)}{2\log q\log l}\right\rceil.
\end{align*}
Table~\ref{taba} lists the results of \( N_{\text{exp}} \)   in  different cases. It is noteworthy that, under the same level of tolerance error, the case with \( (q = 10, \ l = 1.1) \) yields the smallest \( N_{\text{exp}} \).
\begin{table}[ht]
	\centering
	\small
	\caption{Values of \( N_{\text{exp}}=(K+1)J \) for different levels of tolerance error \( \varepsilon \) and different choices of $q,l$.}
	\label{taba}
	\begin{tabular}{cccc}
		\toprule
		\( q, l \) & \( \varepsilon=10^{-2} \) & \( \varepsilon=10^{-3} \) & \( \varepsilon=10^{-4} \) \\
		\midrule
		\( q=2, l=1.5 \) & 88($K=7,\ J=11$) & 176($K=16,\ J=16$) & 315($K=14,\ J=21$) \\
		\( q=8, l=1.1 \) & 64($K=3,\ J=16$) & 115($K=4,\ J=23$) & 174($K=5,\ J=29$) \\
		\( q=9, l=1.1 \) & 60($K=3,\ J=15$) & 110($K=4,\ J=22$) & 168($K=5,\ J=28$) \\
		*\ \( q=10, l=1.1 \) & 42($K=2,\ J=14$) & 84($K=3,\ J=21$) & 135($K=4,\ J=27$) \\
		\( q=11, l=1.09 \) & 45($K=2,\ J=15$) & 88($K=3,\ J=22$) & 140($K=4,\ J=28$) \\
		\bottomrule
	\end{tabular}
\end{table}

Numerical results of   the SOE approximation error $|R_{soe}(t)|$  in different cases are demonstrated in   \cref{fig1}. Note that by   \eqref{approxerrorE}   $R_{soe}(t)$ is of the form
$$R_{soe}(t)= E_{\alpha}(-t^\alpha)-\sum_{j=1}^{N_{exp}}b_je^{-a_jt},$$
and in our actual computation   the term $E_{\alpha}(-t^\alpha)$ is quantified by using the optimal parabolic contour algorithm  \cite{Garrappa2015Numerical}.  
\end{example}

From \cref{fig1} we have the following observations:%
\begin{itemize}
\item   
\cref{alpha02var001,alpha02var0001,alpha02var00001,alpha05var0001,alpha07var0001,samealphadiffvare}  plotted the results of  $R_{soe}(t)$  against $t$.  We can see that for   fixed $\alpha$  and tolerance error $\varepsilon$, the   obtained SOE approximation with different choices of $q$ and $l$ satisfying \eqref{ql-choice} is of the accuracy  $R_{soe}(t)=\mathcal O(\varepsilon)$. This is conformable to the theoretical prediction \eqref{SOEapp2} in \cref{approxerror}.

\item In particular,   the case with \( (q = 10, \ l = 1.1) \) has the smallest \( N_{\textit{exp}} \) among all the cases (cf. Table~\ref{taba}). As far as the complexity is concerned, this is the best choice of $q$ and $l$ in comparison.

\item \cref{lqnotconditionalpha0202,lqnotconditionalpha0203,lqnotconditionalpha0204} also give   results of $R_{soe}(t)$ in the case \( (q = 10, \ l = 2) \)  not satisfying  the condition \eqref{ql-choice}. We can see that  the approximation accuracy in this case is not as good as that in other cases.  

\item  \Cref{samealphadiffvare} shows   results of $R_{soe}(t)$ at different $\alpha$ and $\varepsilon$.
 In the relatively best case \( (q = 10, \ l = 1.1) \).  We can see that for each  \(\alpha\),    the smaller the  tolerance error \(\varepsilon\) becomes, the more accurate the SOE approximation will be. 

\item \Cref{logerror} demonstrates that the logarithmic error of the SOE approximation is proportional to \(|\log(\varepsilon)|\) when \(\alpha\) is fixed. We  can also observe  that   $R_{soe}(t)$ decreases over $t$. This indicates that the SOE approximation is particularly suitable  for long-time simulations in fractional viscoelastic models.

\end{itemize}

\begin{figure}[H]
	\centering 
	\subfigure[$\alpha=0.2$, $\varepsilon=10^{-2}$]{
		\includegraphics[width=.4\linewidth]{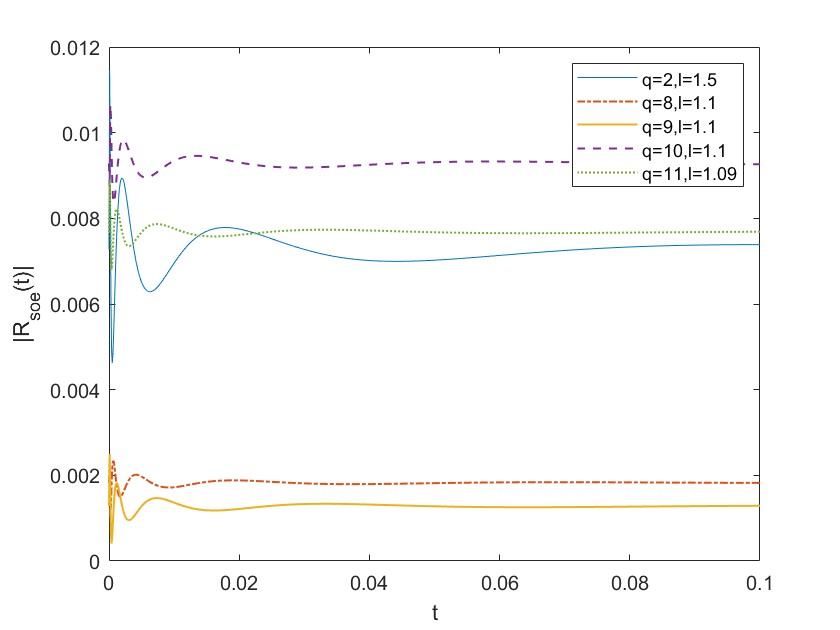}
		\label{alpha02var001}
	}
		\subfigure[$\alpha=0.2$, $\varepsilon=10^{-2}$]{
		\includegraphics[width=.4\linewidth]{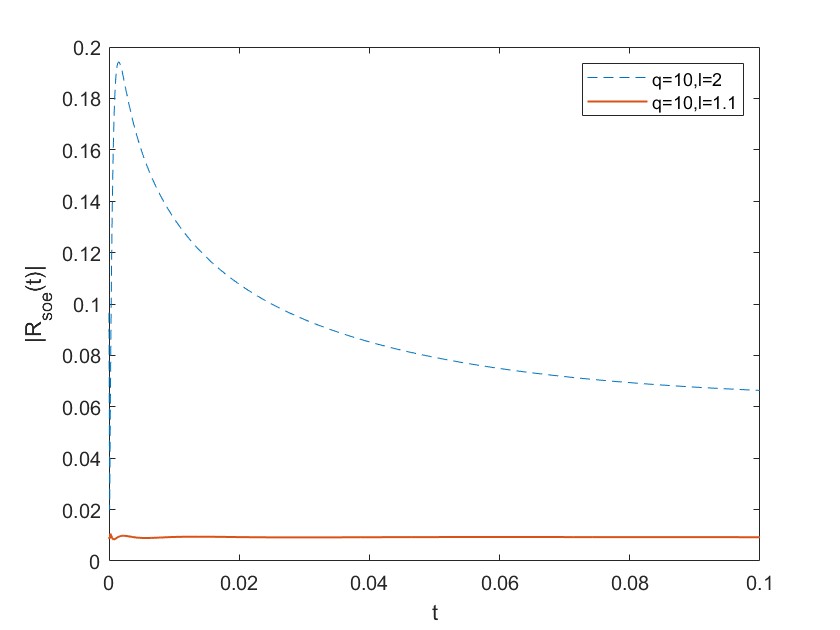}
		\label{lqnotconditionalpha0202}
	}\\
	\centering
	\subfigure[$\alpha=0.2$, $\varepsilon=10^{-3}$]{
		\includegraphics[width=.4\linewidth]{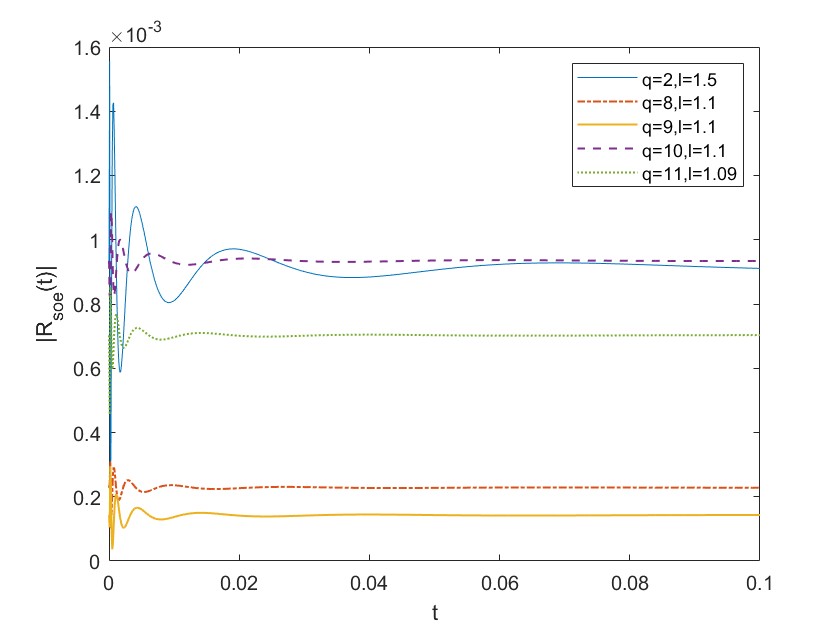}
		\label{alpha02var0001}
	}
	\subfigure[$\alpha=0.2$, $\varepsilon=10^{-3}$]{
		\includegraphics[width=.4\linewidth]{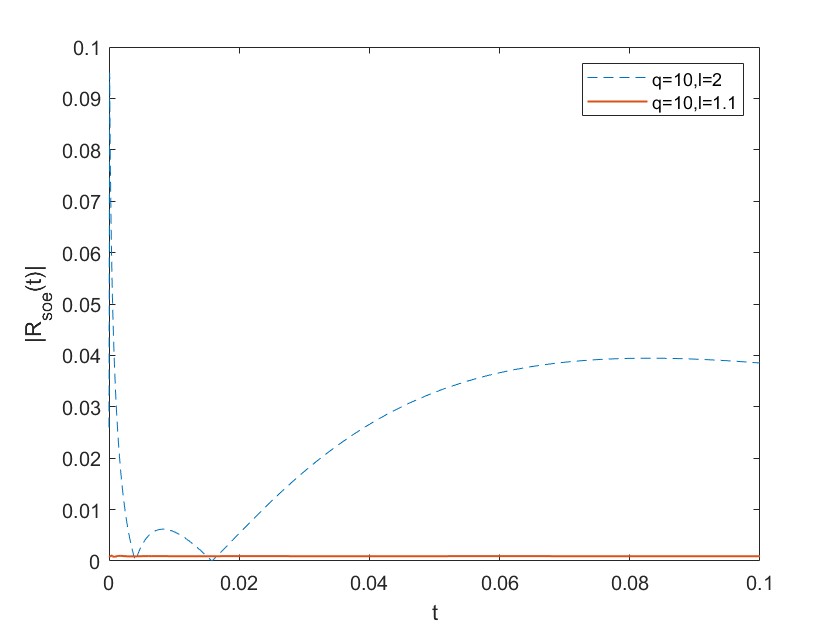}
		\label{lqnotconditionalpha0203}
	}\\
	\centering
	\subfigure[$\alpha=0.2$, $\varepsilon=10^{-4}$]{
		\includegraphics[width=.4\linewidth]{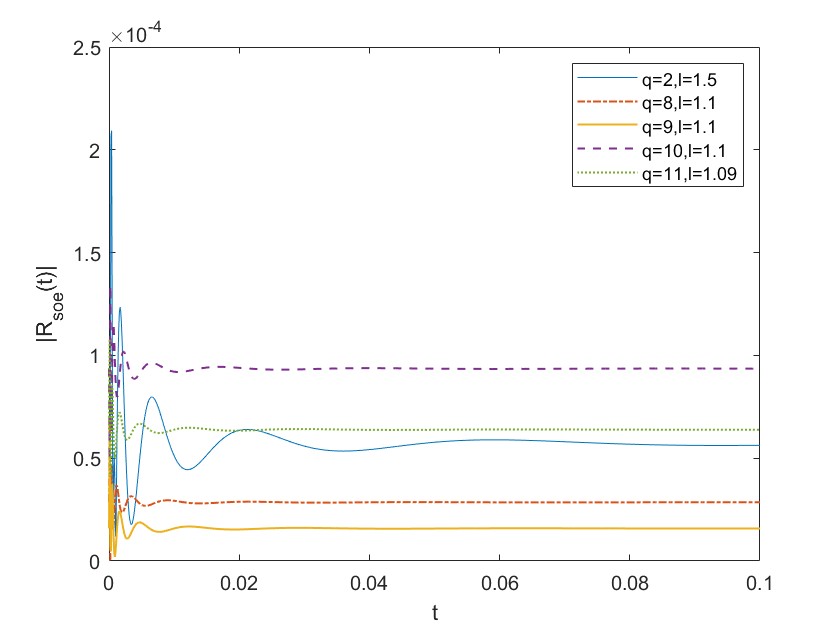}
		\label{alpha02var00001}
	}
	\subfigure[$\alpha=0.2$, $\varepsilon=10^{-4}$]{
		\includegraphics[width=.4\linewidth]{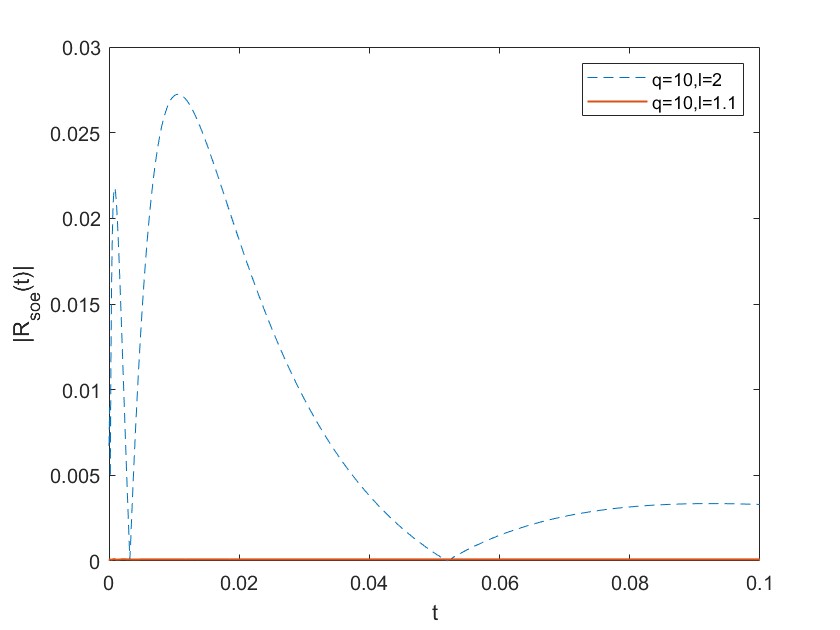}
		\label{lqnotconditionalpha0204}
	}
	\\
	\centering
	\subfigure[$\alpha=0.5$, $\varepsilon=10^{-3}$]{
		\includegraphics[width=.4\linewidth]{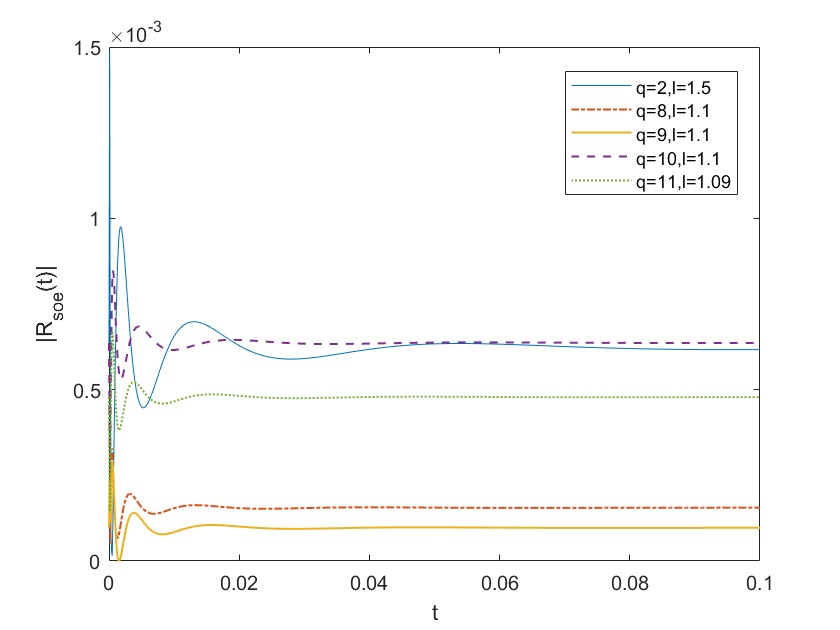}
		\label{alpha05var0001}
	}
	\subfigure[$\alpha=0.7$, $\varepsilon=10^{-3}$]{
		\includegraphics[width=.4\linewidth]{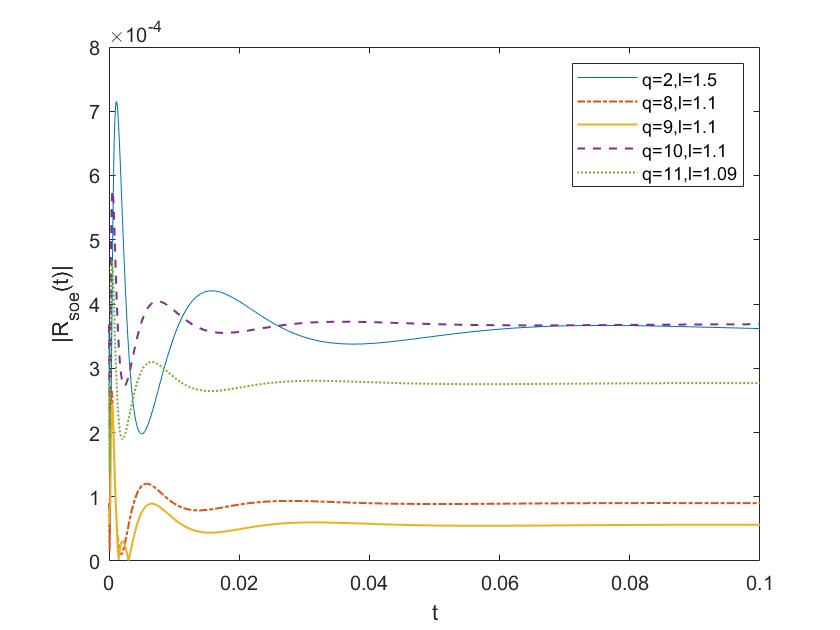}
		\label{alpha07var0001}
	}
	\subfigure[  $q=10$, $l=1.1$]{
		\includegraphics[width=.4\linewidth]{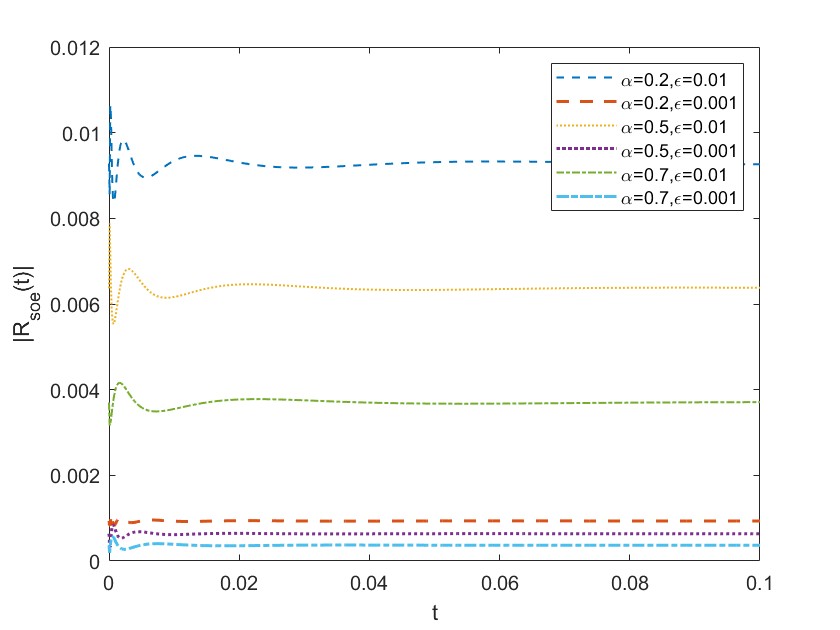}
		\label{samealphadiffvare}
	}
	\subfigure[$\alpha=0.5$, $q=10$, $l=1.1$]{
		\includegraphics[width=.4\linewidth]{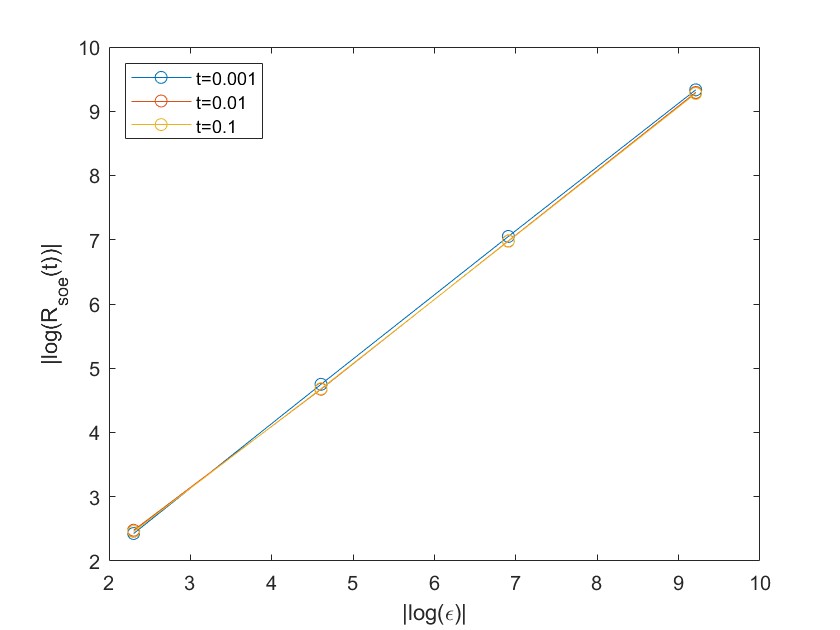}
		\label{logerror}
             }
\caption{ Results of SOE approximation error $R_{soe}(t)$   for   $E_{\alpha}(-t^{\alpha})$. 
    }\label{fig1}	
\end{figure}

\begin{example}[Efficiency test of fast scheme]    In the model problem \cref{model}, we take  \(\Omega = [0, 1] \times [0, 1]\),  \(T = 1\), \(\alpha = 0.5\), \(\tau_{\sigma} = 1\), and \(\tau_{\varepsilon} = 1\).  The elastic medium is assumed to be isotropic, with material properties \(\rho = 1\), \(\mu = 1\), and \(\lambda = 1\), and the exact displacement field $\rm{\bf{u}}(\mathit{x,y,t})$ of the model is also assumed to take the form 
	\rm{\begin{equation*}
			\textbf{u}(x,y,t)=\left(
			\begin{array}{l}
				e^{-t}(x^2-x)^2(4y^3-6y^2+2y)\\
				-e^{-t}(y^2-y)^2(4x^3-6x^2+2x)
			\end{array}
			\right).
	\end{equation*}}

In \cref{algo2} we use  $\frac{1}{h}\times \frac{1}{h}$ square meshes and $N$ uniform grids for   the spatial domain $\Omega$ and the time region $[0,T]$. For the spatial discretization, we apply the Hu-Man-Zhang rectangular element \cite{HuManZhang} spaces, i.e. 
\begin{align*}
 &\uuline{\mathrm{H}_h}=\left\{\chi\in\uuline{\mathrm{H}}(\mathrm{\bf{div}},\Omega,S); \ \chi_{11}\in P_{2,0}(T),\chi_{22}\in P_{0,2}(T), 
 \chi_{12}\in Q_1(T) \  \forall \, T\in\mathcal{T}_h \right\}, \\
 &\uline{\mathrm{V}_h}=\left\{w\in\uline{L}^2(\Omega);\  w_1\in P_{1,0}(T), w_2\in P_{0,1}(T) \  \forall \, T\in\mathcal{T}_h\right\}.
\end{align*}
The local nodal degrees of freedom for the stress tensor  $\tau$ are shown in \cref{Hu-Man-Zhang}.

\begin{figure}[h]
 \centering
 \includegraphics[width=11.5cm,height=4.5cm]{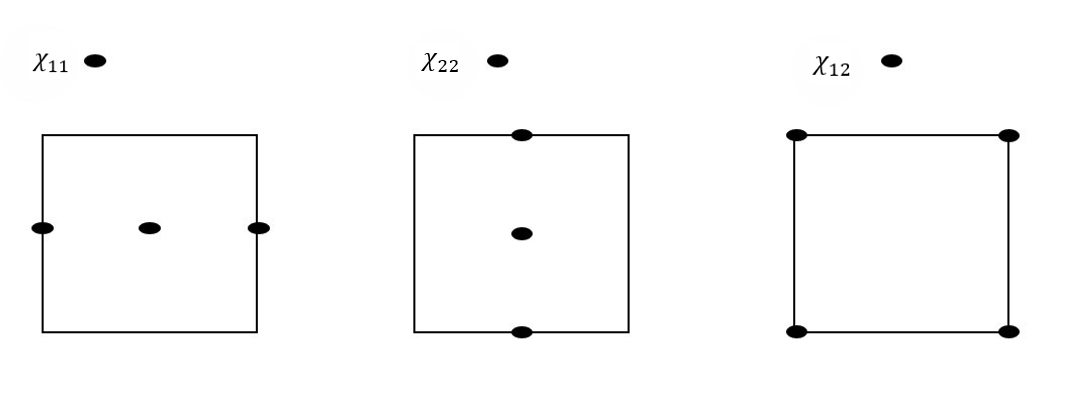}
 \caption{Nodal degrees of freedom for Hu-Man-Zhang's element}
 \label{Hu-Man-Zhang}
\end{figure}
   In the SOE approximation, we set \(q = 10\), \(l = 1.1\), and \(\varepsilon = 10^{-3}\). Numerical results of the error  
\begin{align*}
		||U-\textbf{u}||_{l^\infty}:=\mathop{\mathrm{max}}_{1\leq n\leq N}||U(t_{n})-\textbf{u}(t_{n})||_{L^2(\Omega)},
\end{align*}
the memory cost of the history variables  $H_n$ in \cref{algo1} and $G_j$ in \cref{algo2}, and the wall time of the total runtime of the algorithms are given in Tables \ref{table5}, \ref{table6} and \ref{table7} (All timings are measured in wall-clock seconds using MATLAB's tic/toc mechanism). Here, the wall time does not include   the time required   to solve the linear system.
	\end{example}
	
	From Tables   \ref{table5}, \ref{table6} and  \ref{table7} we can see that when the spatial mesh is fixed, the errors of   the L1-Newmark MFE scheme and the fast scheme   are close. However, the wall time and the memory (Mem) cost of the fast scheme are much less than those of the L1-Newmark MFE scheme. In particular, the memory cost of the fast algorithm remains almost constant as the time step size becomes smaller. This is because the fast algorithm  only requires $N_{exp}$ pieces of information from the $(N-1)$-th level (cf. (\ref{Gin})) to compute the $N$-th time level, whereas the L1-Newmark MFE scheme needs all historical information from the previous $N-1$ time levels (cf. (\ref{His})).  Figure \ref{difftimeandmemory} gives   results of  the wall time and the memory cost of the two algorithms at different time steps when $h=\frac{1}{8}$, showing  that the fast scheme performs excellently.


\begin{table}[H]
	\centering
	\caption{\label{table5}Results of wall time, memory cost and error $||U-\rm{\bf{u}}||_{\textit{l}^\infty}$ for $h=\frac{1}{8}$.}
	\footnotesize 
	\setlength{\tabcolsep}{0.9mm} 
	\sisetup{
		round-mode=places,
		round-precision=7,
		output-exponent-marker = \text{e},
		scientific-notation=true
	}
	\begin{tabular}{
			@{} 
			c 
			c 
			c
			S[table-format=1.7e-2,scientific-notation=true] 
			c
			c
			S[table-format=1.7e-2,scientific-notation=true] 
			@{}
		}
		\toprule
		\multirow{2}{*}{$\triangle t$} & \multicolumn{3}{c}{L1-Newmark}&\multicolumn{3}{c}{Fast Scheme}\\
		\cmidrule(r){2-4} \cmidrule(r){5-7}
		& time (s) & Mem (MB) & {$||U-\textbf{u}||_{l^\infty}$} & time  (s) &  Mem (MB) & {$||U-\textbf{u}||_{l^\infty}$} \\
		\midrule
		$0.01$&0.062&0.411&0.001816104488224&0.002&0.086&0.001815976306702 \\
		$0.005$&0.130&0.821&0.001815941310896&0.004&0.086&0.001815867407206 \\
		$0.001$&0.976&4.098&0.001815957293518&0.026&0.086&0.001815934476067 \\
		$0.0005$&2.777&8.194&0.001815952263475&0.043&0.086&0.001815936540757 \\
		$0.0001$&49.261&40.962&0.001815946841198&0.212&0.086&0.001815937119515 \\
		\bottomrule
	\end{tabular}
\end{table}


\begin{table}[H]
	\centering
	\caption{\label{table6}Results of wall time, memory cost  and error $||U-\rm{\bf{u}}||_{\textit{l}^\infty}$ for  $h=\frac{1}{16}$.}
	\footnotesize 
	\setlength{\tabcolsep}{0.9mm} 
	\sisetup{
		round-mode=places,
		round-precision=7,
		output-exponent-marker = \text{e},
		scientific-notation=true
	}
	\begin{tabular}{
			@{} 
			c 
			c 
			c
			S[table-format=1.7e-2,scientific-notation=true] 
			c
			c
			S[table-format=1.7e-2,scientific-notation=true] 
			@{}
		}
		\toprule
		\multirow{2}{*}{$\triangle t$} & \multicolumn{3}{c}{L1-Newmark}&\multicolumn{3}{c}{Fast Scheme}\\
		\cmidrule(r){2-4} \cmidrule(r){5-7}
		& time (s) & Mem (MB) & {$||U-\textbf{u}||_{l^\infty}$} & time  (s) &  Mem (MB) & {$||U-\textbf{u}||_{l^\infty}$} \\
		\midrule
		$0.01$&0.351&1.647&0.001308384107714&0.034&0.344&0.001308333701266 \\
		$0.005$&0.742&3.285&0.001308378603218&0.067&0.344&0.001308349533648 \\
		$0.001$&5.264&16.392&0.001308429163928&0.343&0.344&0.001308419937633 \\
		$0.0005$&13.839&32.776&0.001308424482868&0.666&0.344&0.001308418084284 \\
		$0.0001$&216.673&163.848&0.001308421273700&3.276&0.344&0.001308417222615\\
		\bottomrule
	\end{tabular}
\end{table}


\begin{table}[H]
	\centering
	\caption{\label{table7}Results of wall time, memory cost  and error $||U-\rm{\bf{u}}||_{\textit{l}^\infty}$ for $h=\frac{1}{32}$.}
	\footnotesize 
	\setlength{\tabcolsep}{0.9mm} 
	\sisetup{
		round-mode=places,
		round-precision=7,
		output-exponent-marker = \text{e},
		scientific-notation=true
	}
	\begin{tabular}{
			@{} 
			c 
			c 
			c
			S[table-format=1.7e-2,scientific-notation=true] 
			c
			c
			S[table-format=1.7e-2,scientific-notation=true] 
			@{}
		}
		\toprule
		\multirow{2}{*}{$\triangle t$} & \multicolumn{3}{c}{L1-Newmark}&\multicolumn{3}{c}{Fast Scheme}\\
		\cmidrule(r){2-4} \cmidrule(r){5-7}
		& time (s) & Mem (MB) & {$||U-\textbf{u}||_{l^\infty}$} & time  (s) &  Mem (MB) & {$||U-\textbf{u}||_{l^\infty}$} \\
		\midrule
		$0.01$&2.824&6.586&0.001140568108041&1.153&1.376&0.001140552874718 \\
		$0.005$&5.688&13.141&0.001140580833534&2.152&1.376&0.001140572259132 \\
		$0.001$&33.486&65.569&0.001140729922473&10.728&1.376&0.001140727198145 \\
		$0.0005$&82.683&131.114&0.001140696786343&21.532&1.376&0.001140694892847 \\
		$0.0001$&1003.425&655.387&0.001140686649151&107.819&1.376&0.001140685770132\\
		\bottomrule
	\end{tabular}
\end{table}

\begin{figure}[H]
	\centering 
	\subfigure{
		\includegraphics[width=.4\linewidth]{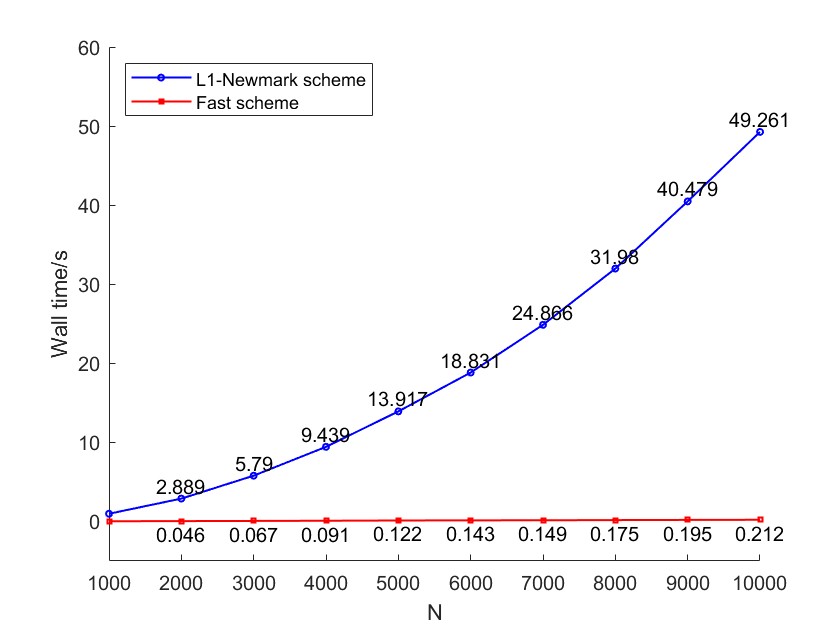}
		\label{difftime}
	}
	\subfigure{
		\includegraphics[width=.4\linewidth]{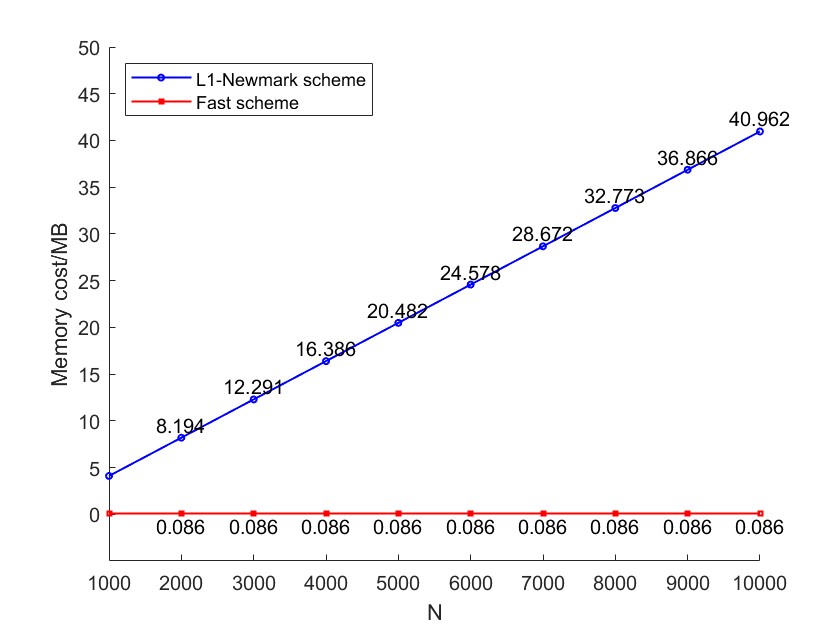}
		\label{diffmemory}
	}
\caption{ A comparison of wall time and memory cost between the two algorithms for different time steps when $h=\frac{1}{8}$ }
\label{difftimeandmemory}	
\end{figure}

\section{Conclusions}
We have proposed an efficient SOE approximation for the Mittag-Leffler function, with a rigorous error analysis, 
 and applied it to develop   a fast fully discrete mixed finite element scheme for   the fractional viscoelastic wave propagation model.   Numerical experiments have demonstrated the excellent performance of the developed  fast scheme  in storage and computing efficiency.

\bibliographystyle{plain}
\normalem
\bibliography{fraviscoelastic}

\begin{thebibliography}{10}

\bibitem{adolfsson2004adaptive}
K.~Adolfsson, M.~Enelund, and S.~Larsson.
\newblock {Adaptive discretization of fractional order viscoelasticity using
  sparse time history}.
\newblock {\em Computer methods in applied mechanics and Engineering},
  193(42-44):4567--4590, 2004.

\bibitem{adolfsson2008space}
K.~Adolfsson, M.~Enelund, and S.~Larsson.
\newblock {Space-time discretization of an integro-differential equation
  modeling quasi-static fractional-order viscoelasticity}.
\newblock {\em Journal of Vibration and Control}, 14(9-10):1631--1649, 2008.

\bibitem{baffet2019gauss}
D.~Baffet.
\newblock {A Gauss--Jacobi kernel compression scheme for fractional
  differential equations}.
\newblock {\em Journal of Scientific Computing}, 79:227--248, 2019.

\bibitem{baffet2017kernel}
D.~Baffet and JS. Hesthaven.
\newblock {A kernel compression scheme for fractional differential equations}.
\newblock {\em SIAM Journal on Numerical Analysis}, 55(2):496--520, 2017.

\bibitem{bagley1983theoretical}
RL. Bagley and PJ. Torvik.
\newblock {A theoretical basis for the application of fractional calculus to
  viscoelasticity}.
\newblock {\em Journal of Rheology}, 27(3):201--210, 1983.

\bibitem{bagley1981fractional}
RL. Bagley and PJ. Torvik.
\newblock {Fractional calculus-A different approach to the finite element
  analysis of viscoelastically damped structures}.
\newblock {\em AIAA Journal}, 21(5):741--748, 1983.

\bibitem{bai2022numerical}
X.~Bai, J.~Huang, H.~Rui, and S.~Wang.
\newblock {Numerical simulation for 2{D}/3{D} time fractional Maxwell's system
  based on a fast second-order FDTD algorithm}.
\newblock {\em Journal of Computational and Applied Mathematics}, 416, 2022.

\bibitem{1983Computational}
T.~Belytschko and T.~J.~R. Hughes.
\newblock {A Precis of Developments in Computational Methods for Transient
  Analysis}.
\newblock {\em Journal of Applied Mechanics-Transactiions of the Asme},
  50(4B):1033--1041, 1983.

\bibitem{blair1944analytical}
GW.~S. Blair.
\newblock {Analytical and Integrative Aspects of the Stress-Strain-Time
  Problem}.
\newblock {\em Journal of Scientific Instruments}, 21(5):80, 2002.

\bibitem{1960Bland}
D.~R. Bland.
\newblock {\em {The theory of linear viscoelasticity}}.
\newblock International Series of Monographs on Pure and Applied Mathematics.
  Pergamon Press, 1960.

\bibitem{caputo1971new}
M.~Caputo and F.~Mainardi.
\newblock {A new dissipation model based on memory mechanism}.
\newblock {\em Pure and applied Geophysics}, 91:134--147, 1971.

\bibitem{caputo1971linear}
M.~Caputo and F.~Mainardi.
\newblock {Linear models of dissipation in anelastic solids}.
\newblock {\em La Rivista del Nuovo Cimento}, 1(2):161--198, 1971.

\bibitem{chen2019accurate}
L.~Chen, J.~Zhang, J.~Zhao, W.~Cao, H.~Wang, and J.~Zhang.
\newblock {An accurate and efficient algorithm for the time-fractional
  molecular beam epitaxy model with slope selection}.
\newblock {\em Computer Physics Communications}, 245, 2019.

\bibitem{Diethelm2010}
K.~Diethelm.
\newblock {\em {The Analysis of Fractional Differential Equations}}.
\newblock Springer, Berlin, 2010.

\bibitem{2007Dill}
E.H. Dill.
\newblock {\em {Continuum Mechanics : Elasticity, Plasticity,
  Viscoelasticity}}.
\newblock CRC Press, 2007.

\bibitem{1998Drozdov}
A.~D. Drozdov.
\newblock {\em {Mechanics of Viscoelastic Solids}}.
\newblock Wiley, 1998.

\bibitem{enelund1997time}
M.~Enelund and B.~L. Josefson.
\newblock {Time-domain finite element analysis of viscoelastic structures with
  fractional derivatives constitutive relations}.
\newblock {\em AIAA Journal}, 35(10):1630--1637, 1997.

\bibitem{fang2020fast}
Z.~Fang, H.~Sun, and H.~Wang.
\newblock {A fast method for variable-order Caputo fractional derivative with
  applications to time-fractional diffusion equations}.
\newblock {\em Computers \& Mathematics with Applications}, 80(5):1443--1458,
  2020.

\bibitem{Fung1966International}
Y.~C Fung.
\newblock {International Series on Dynamics. (Book Reviews: Foundations of
  Solid Mechanics)}.
\newblock {\em Science}, 152, 1966.

\bibitem{gautschi1981survey}
W.~Gautschi.
\newblock {A survey of Gauss-Christoffel quadrature formulae}.
\newblock {\em in EB Christoffel: The influence of his work on mathematics and
  the physical sciences,P.L.Butzer and F.Feh\'{e}r, eds.,Birk\"{a}user,Basel},
  pages 72--147, 1981.

\bibitem{gemant1936method}
A.~Gemant.
\newblock {A method of analyzing experimental results obtained from
  elasto-viscous bodies}.
\newblock {\em Physics}, 7(8):311--317, 1936.

\bibitem{gemant1950frictional}
A.~Gemant.
\newblock {Frictional phenomena}.
\newblock {\em Journal of Applied Physics}, 13(5):290--299, 1942.

\bibitem{1988Boundary}
J.~M. Golden and G.~A.~C. Graham.
\newblock {\em {Boundary Value Problems in Linear Viscoelasticity}}.
\newblock Springer, 1988.

\bibitem{2014Rogosin}
R.~Gorenflo, A.~A. Kilbas, F.~Mainardi, and S.~V. Rogosin.
\newblock {\em Mittag-Leffler Functions, Related Topics and Applications}.
\newblock Springer, Heidelberg, 2014.

\bibitem{1962Gurtin}
M.~E. Gurtin and E.~Sternberg.
\newblock {On the linear theory of viscoelasticity}.
\newblock {\em Archive for Rational Mechanics and Analysis}, 11(1):291--356,
  1962.

\bibitem{HuManZhang}
J.~Hu, H.~Man, and S.~Zhang.
\newblock {A Simple Conforming Mixed Finite Element for Linear Elasticity on
  Rectangular Grids in Any Space Dimension}.
\newblock {\em Journal of Scientific Computing}, 58(2):367--379, 2014.

\bibitem{2022Huang}
Y.~Huang, Q.~Li, R.~Li, F.~Zeng, and L~Guo.
\newblock {A Unified Fast Memory-Saving Time-Stepping Method for Fractional
  Operators and Its Applications}.
\newblock {\em Numerical Mathematics-Theory Methods And Applications},
  15(3):679--714, 2022.

\bibitem{jia2023fast}
J.~Jia, H.~Wang, and X.~Zheng.
\newblock {A fast algorithm for time-fractional diffusion equation with
  space-time-dependent variable order}.
\newblock {\em Numerical Algorithms}, 94:1705--1730, 2023.

\bibitem{jiang2017fast}
S.~Jiang, J.~Zhang, Q.~Zhang, and Z.~Zhang.
\newblock {Fast evaluation of the Caputo fractional derivative and its
  applications to fractional diffusion equations}.
\newblock {\em Communications in Computational Physics}, 21(3):650--678, 2017.

\bibitem{BangtiJin2021}
B.~Jin.
\newblock {\em {Fractional Differential Equations}}.
\newblock Springer, 2021.

\bibitem{Podlubny1999}
l.~Podlubny.
\newblock {\em Fractional Differential Equations: An Introduction to
  Fractiorlal Derivatives, Fractiorlal Differential Eqnations, to Methods of
  their Solutiori and some of their Applications}.
\newblock Academic Press, 1999.

\bibitem{lam2020exponential}
PH. Lam, HC. So, and CF. Chan.
\newblock Exponential sum approximation for mittag-leffler function and its
  application to fractional zener wave equation.
\newblock {\em Journal of Computational Physics}, 410:109389, 2020.

\bibitem{li2010fast}
J.~Li.
\newblock {A fast time stepping method for evaluating fractional integrals}.
\newblock {\em SIAM Journal on Scientific Computing}, 31(6):4696--4714, 2010.

\bibitem{2024Liu-Xie}
M.~Liu and X.~Xie.
\newblock A hybrid stress finite element method for integro-differential
  equations modelling dynamic fractional order viscoelasticity.
\newblock {\em Int. J. Numer. Anal. Mod.}, 2024.

\bibitem{lubich2002fast}
C.~Lubich and A.~Sch{\"a}dle.
\newblock {Fast convolution for nonreflecting boundary conditions}.
\newblock {\em SIAM Journal on Scientific Computing}, 24(1):161--182, 2002.

\bibitem{mainardi2022fractional}
F.~Mainardi.
\newblock {\em {Fractional calculus and waves in linear viscoelasticity: an
  introduction to mathematical models}}.
\newblock World Scientific, 2022.

\bibitem{meng2018green}
X.~Meng and M.~Stynes.
\newblock {The Green's function and a maximum principle for a Caputo two-point
  boundary value problem with a convection term}.
\newblock {\em Journal of Mathematical Analysis and Applications},
  461(1):198--218, 2018.

\bibitem{newmark1959method}
N.~M. Newmark.
\newblock {A method of computation for structural dynamics}.
\newblock {\em Journal of the engineering mechanics division}, 85(3):67--94,
  1959.

\bibitem{Garrappa2015Numerical}
Garrappa. R.
\newblock {Numerical evaluation of two and three parameter Mittag-Leffler
  functions}.
\newblock {\em Siam Journal on Numerical Analysis}, 53(3):26--37, 2015.

\bibitem{rabotnov1970creep}
YN. Rabotnov, FA. Leckie, and W.~Prager.
\newblock {Creep Problems in Structural Members}.
\newblock {\em Journal of Applied Mechanics}, 37(1):249, 1970.

\bibitem{2000Nonlinear}
R.~A Schapery.
\newblock {Nonlinear viscoelastic solids}.
\newblock {\em International Journal of Solids and Structures},
  37(1--2):359--366, 2000.

\bibitem{Marques2012Computational}
P.~C.~M. Severino and J.~C. Guillermo.
\newblock {\em { Computational Viscoelasticity }}.
\newblock Springer New York, 2012.

\bibitem{wang2018fast}
K.~Wang and J.~Huang.
\newblock {A fast algorithm for the Caputo fractional derivative}.
\newblock {\em East Asian J. Appl. Math}, 8(4):656--677, 2018.

\bibitem{yan2017fast}
Y.~Yan, Z.~Sun, and J.~Zhang.
\newblock {Fast evaluation of the Caputo fractional derivative and its
  applications to fractional diffusion equations: a second-order scheme}.
\newblock {\em Communications in Computational Physics}, 22(4):1028--1048,
  2017.

\bibitem{yin2021class}
B.~Yin, Y.~Liu, H.~Li, and F.~Zeng.
\newblock {A class of efficient time-stepping methods for multi-term
  time-fractional reaction-diffusion-wave equations}.
\newblock {\em Applied Numerical Mathematics}, 165:56--82, 2021.

\bibitem{yu2016fractional}
Y.~Yu, P.~Perdikaris, and GE. Karniadakis.
\newblock {Fractional modeling of viscoelasticity in 3D cerebral arteries and
  aneurysms}.
\newblock {\em Journal of computational physics}, 323:219--242, 2016.

\bibitem{zeng2018stable}
F.~Zeng, I.~Turner, and K.~Burrage.
\newblock {A stable fast time-stepping method for fractional integral and
  derivative operators}.
\newblock {\em Journal of Scientific Computing}, 77:283--307, 2018.

\bibitem{zhang2022exponential}
J.~Zhang, Z.~Fang, and H.~Sun.
\newblock {Exponential-sum-approximation technique for variable-order
  time-fractional diffusion equations}.
\newblock {\em Journal of Applied Mathematics and Computing}, 68(1):323--347,
  2022.

\end{thebibliography}

\end{document}